%% file: main.tex
\documentclass[a4paper]{article}

\input{Defines.tex}

\title{Exact Decomposition Branching exploiting Lattice Structures}

\newcommand{\myorcidlink}[1]{\,\href{https://orcid.org/#1}{\raisebox{-0.45ex}{\includegraphics[width=1.8ex]{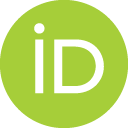}}}}

\author{
	Katrin Halbig\protect\myorcidlink{0000-0002-8730-3447}\footnote{Corresponding author, katrin.halbig@fau.de}
	\and Timm Oertel\protect\myorcidlink{0000-0001-5720-8978}
	\and Dieter Weninger\protect\myorcidlink{0000-0002-1333-8591}
}

\date{{\itshape Friedrich-Alexander-Universität Erlangen-Nürnberg,\\Cauerstr.~11, 91058 Erlangen, Germany}\\\vspace*{2ex}
	October, 2024}

\begin{document}

\maketitle

\input{section_abstract}

\input{section_introduction}
\input{section_roundingrules}
\input{section_misl}

\input{section_decbranching}

\input{section_results}
\input{section_conclusion}

\bibliographystyle{abbrvnat}
\bibliography{bibliography}

\end{document}

%% file: Defines.tex
\usepackage[utf8]{inputenc}
\usepackage{amsmath}
\usepackage{amssymb}
\usepackage{amsfonts}
\usepackage{mathtools}
\usepackage{amsthm}
\usepackage{enumitem}
\usepackage{graphicx}
\usepackage{tabularx}
\usepackage{tabularcalc}
\usepackage{booktabs}
\usepackage{multirow}
\usepackage[ruled]{algorithm2e}
\usepackage{subcaption}
\usepackage{siunitx}
\usepackage[breaklinks, colorlinks, linkcolor=blue, citecolor=blue, urlcolor=black]{hyperref}
\usepackage{doi}

\usepackage[numbers]{natbib}

\usepackage{xspace}
\usepackage[table]{xcolor}
\usepackage{collcell}
\usepackage{xspace}

\usepackage{tikz}
\usepackage{pgfplots}
\pgfplotsset{compat=1.16}

\usepackage{todonotes}

\newcommand{\st}{\text{s.t.\xspace}}

\newcommand{\foralltext}{\text{for all\ }}

\newtheorem{theorem}{Theorem}[section]
\newtheorem{assumption}[theorem]{Assumption}
\newtheorem{Corollary}[theorem]{Corollary}
\newtheorem{Lemma}[theorem]{Lemma}

\newtheorem{definition}[theorem]{Definition}

\newtheorem{Remark}[theorem]{Remark}

\DeclareMathOperator{\lcm}{lcm}
\DeclareMathOperator{\diag}{diag}

\newcommand{\testset}[1]{\textsc{#1}}

%% file: section_abstract.tex

\begin{abstract}
Strict inequalities in mixed-integer linear optimization can cause difficulties in guaranteeing convergence and exactness.
Utilizing that optimal vertex solutions follow a lattice structure we propose a rounding rule for strict inequalities that guaranties exactness.
The lattice used is generated by $\Delta$-regularity of the constraint matrix belonging to the continuous variables.
We apply this rounding rule to Decomposition Branching by Yıldız et al., which uses strict inequalities in its branching rule.
We prove that the enhanced algorithm terminates after finite many steps with an exact solution.
To validate our approach, we conduct computational experiments for two different models for which $\Delta$-regularity is easily detectable.
The results confirm the exactness of our enhanced algorithm and demonstrate that it typically generates smaller branch-and-bound trees.
\end{abstract}

{\noindent{\bf Keywords:}
Mixed-integer programming $\cdot$
Decomposition $\cdot$
Branch-and-bound$\cdot$
Lattices $\cdot$
Supply chain management}

%% file: section_introduction.tex
\section{Introduction}
\label{sect:introduction}

Consider a mixed-integer linear problem of the form
\begin{equation}\label{eq:MIP}
\underset{x,y}{\min} \left\{c^\top x + d^\top y \,\vert\, A x + B y \geq g,\, x \geq 0,\, y \geq 0,\, y \in \mathbb{Z}^{\ell} \right\},
\end{equation}
where $c \in \mathbb{Z}^{n}$, $d \in \mathbb{Z}^{\ell}$, $A \in\mathbb{Z}^{m \times n}$, $B \in\mathbb{Z}^{m \times \ell}$, $g \in\mathbb{Z}^m$.
We assume that problem~\eqref{eq:MIP} has an optimal vertex solution, say $(x^*,y^*)$.
Branch-and-bound~\cite{LandDoig:BranchAndBound}
has become the standard paradigm for solving such problems.
Here branching refers to the way the search space is partitioned,
while bounding refers to the process of using valid bounds
to discard branches and thus restricting the search space without compromising optimality.
While usually branch-and-bound type methods use closed half-spaces specified by linear inequalities to partition the search space
there are less common approaches, such as \emph{Decomposition Branching}~\cite{Yildiz_et_al:2022},
that use open half-spaces specified by strict linear inequalities.
Note that there are also numerous heuristics based on open half-spaces, such as~\cite{Berthold2014,FisMon2014}.
However, strict inequalities come with the problem on how to implement them.
An obvious and typical approach is to reformulate a strict inequality such as for example $u^\top x > \gamma$ (with $u\in\mathbb{Z}^{n}$) by $u^\top x \geq \gamma+\varepsilon$ with a sufficiently small $\varepsilon > 0$ (see, e.g.,~\cite{williams2013model} page 170).
But if $\varepsilon$ is not chosen appropriately the inequality might become invalid for $(x^*,y^*)$.

Our main contribution is a rounding procedure for strict inequalities that only takes the matrix $A$ into account, circumventing the later.
Note that for any feasible solution $(\bar{x},\bar{y})$ we have that $g - B \bar{y}$ is integral.
So it suffices to ensure that the rounding procedure does not cut off any vertices of polyeder $P(A,b)\coloneqq \{x \,\vert\, Ax \geq b,\, x \geq 0\}$ where $b \in \mathbb{Z}^m$.
Naturally the vertices follow a lattice structure which is closely related to $\Delta$-regularity of the matrix $A$.
In Section~\ref{sect:lattices} we show how to exploit this.
In Section~\ref{sect:misl} we show how to determine the $\Delta$-regularity for facility location and lot-sizing problems efficiently.
In Section~\ref{sect:decbranching} we show how the methods developed can be used to
dispense with strict inequalities in Decomposition Branching.
Finally, in Section~\ref{sect:results}, we present computational results.

%% file: section_roundingrules.tex
\section{Rounding Approach based on Lattices}
\label{sect:lattices}

In this section we propose a rounding rule for adding strict inequalities to a
polyhedron $P(A,b) \coloneqq \{x \,\vert\, Ax \geq b,\, x\geq0\}$ with $A \in \mathbb{Z}^{m \times n}$ and $b \in \mathbb{Z}^m$.
Let the symbol $I$ denote the identity matrix of appropriate size.

As is well known, every vertex of $P(A,b)$ corresponds to at least one basis, that is a $(n \times n)$-submatrix $D$ of $\begin{pmatrix}A^\top & I\end{pmatrix}^\top$.
Each regular $(n \times n)$-submatrix $D$
generates a lattice $\Lambda_{D^{-1}} = \{D^{-1} \lambda \,\vert\, \lambda \in \mathbb{Z}^n\}$. 
Define $\Lambda$ as the Minkowski sum of all lattices $\Lambda_{D^{-1}}$, which again is a lattice.
Note that any vertex of $P(A,b)$---independent of $b$---is an element of $\Lambda$.
The main observation we use is that any vector $v$ of the dual lattice $\Lambda^* \coloneqq \{y \in \mathbb{Q}^{n} \,\vert\, x^\top y \in \mathbb{Z}\ \foralltext x \in \Lambda\}$ and any $\gamma \in \mathbb{Z}$ defines a split-cut $S=\{x \,|\, \gamma \le v^\top x \le \gamma + 1 \}$ which---by definition---is a lattice-free set, that is $\operatorname{int}(S) \cap \Lambda =\emptyset$.
For a survey on lattice-free sets, see, for example, \cite{MR1114315}.

Now consider a strict inequality
\begin{equation}\label{Eq:strictIneq}
  u^\top x > \gamma,\quad u\in\mathbb{Z}^n,\gamma\in\mathbb{Q}.
\end{equation}
As $\Lambda$ and $u$ are rational or integer there exits a $\mu > 0$ with $\mu u \in \Lambda^*$.
Noting that $\{ x\in\Lambda \,|\, u^\top x > \gamma \} = \{ x\in\Lambda \,|\, u^\top x \geq \frac{\lfloor \mu \gamma \rfloor + 1}{\mu} \}$
since split-cut $S$ is lattice-free,
we can now replace the strict inequality by
\begin{equation}\label{Ineq:LatticeApproach1}
  u^\top x \geq \frac{\lfloor \mu \gamma \rfloor + 1}{\mu}.
\end{equation}
Inequality \eqref{Ineq:LatticeApproach1}  has the property that
no points that correspond to vertices of $P(A,b)$ and fulfill the strict inequality are cut off. 

In order to simplify the lattice $\Lambda$, to reduce the number of parameters, and to omit solving a minimization problem to determine $\mu$, we want to replace in a next step $\Lambda$ by a superset of the form $\tfrac{1}{\Delta} \mathbb{Z}^n$.
Which leads us to a notion of regularity which has been introduced by \citeauthor{Appa_Kotnyek_2004} in~\cite{Appa_Kotnyek_2004}.
A rational matrix is called \emph{$\Delta$-regular} for $\Delta \in \mathbb{N}$ if for each of its non-singular square submatrices $R$, $\Delta R^{-1}$ is integral.
Let us denote lattice $\Lambda_\Delta \coloneqq \tfrac{1}{\Delta} \mathbb{Z}^n = \{\frac{1}{\Delta}I \lambda \,\vert\, \lambda \in \mathbb{Z}^n\}$.
In \cite[Theorem 17]{Appa_Kotnyek_2004} it is stated that all vertices of a polyhedron $P(A,\Delta b)$
are integral for each integral vector $b$ if and only if $A$ is $\Delta$-regular.
Thus, it is not hard to see that if $A$ is $\Delta$-regular, then $\Lambda\subseteq\Lambda_\Delta$ holds.
Now, it follows that $\Lambda^*\supseteq\Lambda_\Delta^*=\Delta\mathbb{Z}^n$.
Considering again the strict inequality~\eqref{Eq:strictIneq}
it holds that $\Delta u \in \Lambda_\Delta^{*}$
and inequality~\eqref{Ineq:LatticeApproach1} becomes simply
\begin{equation}\label{Ineq:LatticeApproach2}
  u^\top x \geq \frac{\lfloor \Delta \gamma \rfloor + 1}{\Delta}.
\end{equation}

Noting that if a rational matrix is $\Delta$-regular, then this matrix is also $(s\cdot\Delta)$-regular, for any $s\in\mathbb{N}$, we are particularly interested in minimal values and introduce the following definition.
\begin{definition}[Minimal $\Delta$-Regularity]\label{Definition:DeltaRegular:minimal}
  A rational matrix is called \emph{minimal $\Delta$-regular}
  if it is $\Delta$-regular and if there is no $s\in \mathbb{N}$ with $s < \Delta$ such that the matrix is $s$-regular.
\end{definition}

To illustrate the overall approach, consider $P(A,b)=\{x \,\vert\, Ax \geq b,\, x\geq0 \}$ with
  \begin{equation}
    A= 
    \begin{pmatrix}
      -1 & -1 \\ -1 & 1 \\ 0 & -2 
    \end{pmatrix} \quad\text{and}\quad b= \begin{pmatrix} -5 \\ -2 \\ -5 \end{pmatrix}.\nonumber
  \end{equation}
  Matrix $A$ is minimal 2-regular.
  The polyhedron $P(A,b)$ is shown in gray in Figure~\ref{fig:example-rounding}.
  It can be seen that for each vertex $\bar{x}$ of $P(A,b)$ holds the property $\bar{x} \in \Lambda_\Delta = \frac{1}{2} \mathbb{Z}^2$.
  For graphical clarification, the intersection points of the dotted lines are the set $\Lambda_\Delta$.  
\begin{figure}[h]
    \begin{center}
        \pgfplotsset{
			rounding plot/.style = {
			unit vector ratio={1 1},
			axis x line=center,
			axis y line=center,
			xtick={1,2,3,4,5},
			ytick={1,2,3},
			xticklabels={1,2,3,4},
			yticklabels={1,2,3},
			xlabel style={below},
			ylabel style={left},
			xmin=-0.2, xmax=4.5,
			ymin=-0.2, ymax=3.5,
			scale=0.85
			},
		}
		\begin{tikzpicture}
		\begin{axis}[rounding plot,
		xlabel={$x_1$},
		ylabel={$x_2$}
		]
			\draw[dotted, line width = 0.2mm] (0,0.5) -- (4,0.5);
			\draw[dotted, line width = 0.2mm] (0,1) -- (4,1);
			\draw[dotted, line width = 0.2mm] (0,1.5) -- (4,1.5);
			\draw[dotted, line width = 0.2mm] (0,2) -- (4,2);
			\draw[dotted, line width = 0.2mm] (0,2.5) -- (4,2.5);
			\draw[dotted, line width = 0.2mm] (0,3) -- (4,3);

			\draw[dotted, line width = 0.2mm] (0.5,0) -- (0.5,3);
			\draw[dotted, line width = 0.2mm] (1,0) -- (1,3);
			\draw[dotted, line width = 0.2mm] (1.5,0) -- (1.5,3);
			\draw[dotted, line width = 0.2mm] (2,0) -- (2,3);
			\draw[dotted, line width = 0.2mm] (2.5,0) -- (2.5,3);
			\draw[dotted, line width = 0.2mm] (3,0) -- (3,3);
			\draw[dotted, line width = 0.2mm] (3.5,0) -- (3.5,3);
			\draw[dotted, line width = 0.2mm] (4,0) -- (4,3);
			
			\draw[line width = 0.4mm] (0,0) -- (2,0) -- (3.5,1.5) -- (2.5,2.5) -- (0,2.5) -- (0,0);
			\fill[fill=gray!60, fill opacity=0.2] (0,0) -- (2,0) -- (3.5,1.5) -- (2.5,2.5) -- (0,2.5) -- (0,0);
			
			\draw[dashed, line width = 0.2mm] (0,2.275) -- (4,1.275);
			\draw[line width = 0.2mm] (0,2.375) -- (4,1.375);
		\end{axis}
		\end{tikzpicture}
		\caption{Rounding of strict inequalities based on $\Delta$-regularity.}
		\label{fig:example-rounding}
	\end{center}
\end{figure}
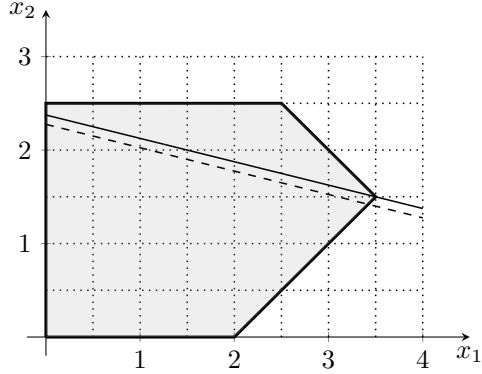

  Now we want to add the strict inequality $x_1 + 4 x_2 > \frac{91}{10}$ (plotted dashed in Figure~\ref{fig:example-rounding}),
  to exclude the half-space $x_1 + 4 x_2 \leq \frac{91}{10}$.
  We can use \eqref{Ineq:LatticeApproach2} with $\Delta=2$
  to obtain the inequality $x_1 + 4 x_2 \geq \frac{19}{2}$,
  which is drawn solid in Figure~\ref{fig:example-rounding}.

%% file: section_misl.tex
\section{Determination of $\Delta$-Regularity}
\label{sect:misl}

In this section, we will deal with the determination of the
minimal $\Delta$-regularity of an integer matrix.
First, we take a look at similar concepts to $\Delta$-regularity and on approaches for determining $\Delta$-regularity from literature.
Then, we will derive upper and lower bounds on the minimal $\Delta$-regularity and, finally,
we will show how the minimal $\Delta$-regularity can be determined as a function of the coefficients for three different models.

$\Delta$-regularity is a generalization of \emph{totally unimodularity}.
A matrix $A$ is totally unimodular if each square submatrix $R$ of $A$ satisfies $\det(R) \in \{0,\pm 1\}$; see~\cite{Schrijver1986} for further details.
It holds that a matrix is $1$-regular if it is totally unimodular.
The set of $1$-regular matrices, however, is larger than that of totally unimodular matrices as a $1$-regular matrix is not necessarily integral.
For example, $A = \begin{pmatrix}\frac{1}{2}\end{pmatrix}$ is a $1$-regular matrix.
However, if $A$ is an integer matrix, we obtain an equivalence statement:
an integral matrix is 1-regular if and only if it is totally unimodular~\cite{Kotnyek2002}.
Based on Seymour's decomposition theorem for regular matroids~\cite{Seymour1980} Truemper~\cite{Truemper1990} developed a highly sophisticated algorithm that decides in polynomial time whether matrix $A$ is totally unimodular with a time complexity of $O((m+n)^3)$ for a $(m \times n)$-matrix.
An implementation of a simplified version of this algorithm was published in~\cite{WalterTruemper2011}.

A similar concept is \emph{totally $\Delta$-modularity}~(see,~e.g.,~\cite{Artmann_etal_2017,Fiorini_etal_2021}).
An integer matrix $A$ is called totally $\Delta$-modular
if for all square submatrices $R$ of $A$ holds $\det(R)\in \{-\Delta,\ldots,\Delta\}$.
If the $\Delta$ is minimal for a matrix $A$, then we call $A$ to be \emph{minimal totally $\Delta$-modular}.
Observe that the minimal totally $\Delta$-modularity can differ significantly from the minimal $\Delta$-regularity.
Consider the two matrices
\[
A_1 = \begin{pmatrix} 5 & 0 & 0 \\
    0 & 5 & 0 \\
    0 & 0 & 1 \end{pmatrix}\quad \text{and}\quad A_2= \begin{pmatrix} 1 & 0 & 0 \\ 0 & 2 & 3 \end{pmatrix}.
\]
Matrix $A_1$ is minimal totally $25$-modular, whereas the matrix is minimal $5$-regular.
In contrast, matrix $A_2$ is minimal totally $3$-modular,
but it is minimal $6$-regular.

A special variant thereof is $\{a, b, c\}$-modularity.
A matrix $A\in\mathbb{Z}^{m \times n}$ is $\{a, b, c\}$-modular
if the set of $(n\times n)$-subdeterminants of
$A$ in absolute value is in $\{a, b, c\}$ with $a,b,c \in\mathbb{Z}$
and $a \geq b \geq c \geq 0$.
In~\cite{Glanzer2022} the authors
succeeded in recognizing $\{a, b, c\}$-modularity in polynomial time
unless $A$ possesses a duplicative relation, that is, $A$ has nonzero $(n \times n)$-subdeterminants $k_1$
and $k_2$ satisfying $2 \vert k_1 \vert = \vert k_2 \vert$.

Focusing again on $\Delta$-regularity, there is an efficient approach for recognizing
so-called \emph{binet matrices}.
A matrix $A$ is called a binet matrix if there exist
both an integral matrix $M$ of full row rank satisfying
$\sum_{j=1}^n \vert M_{ij} \vert \leq 2$ for any column index $j$
and a basis $B$ of it such that $M = \begin{pmatrix}B & N \end{pmatrix}$ (up to column permutation) and $A = B^{-1} N$.
Binet matrices are 2-regular and a generalization of network matrices.
The main result in \cite{Musitelli_2010} is a polynomial-time algorithm for the recognition of binet matrices and thus for the recognition of a subclass of 2-regular matrices.

For an arbitrary matrix $A \in \mathbb{Z}^{m \times n}$, minimal $\Delta$-regularity can be determined by a trivial brute force approach.
One can invert all non-singular square submatrices~$R \in \mathbb{Z}^{w \times w}$
of $A$ with $w \leq \min(m,n)$ and calculate for each $R^{-1}$ the least common multiple of all denominators of the entries of $R^{-1}$ to determine a minimal integer number $\Delta_R$ such that $\Delta_R R^{-1} \in \mathbb{Z}$.
The least common multiple of all $\Delta_R$ is the minimal $\Delta$-regularity of $A$.
However, this can be even more challenging than solving the corresponding optimization problem~\eqref{eq:MIP} by itself.
Whereas there exist algorithms for calculating the matrix inversion in $O(w^3)$,
one has to check up to
$\sum_{w=1}^{\min(m,n)}\binom{m}{w}\binom{n}{w}$
submatrices for determining the minimal $\Delta$-regularity of $A$.

As there is currently no polynomial-time algorithm for the determination of the minimal $\Delta$-regularity of an arbitrary matrix $A$ known so far,
it is of central importance to be able to specify efficiently determinable bounds.
We can state a lower bound on $\Delta$, which follows directly from the definition of $\Delta$-regularity by considering only $(1\times1)$-submatrices:
\begin{Lemma}[Lower Bound Minimal $\Delta$-Regularity]\label{Lemma:lowerbound}
	Let $A$ be an integral matrix.
	Its minimal $\Delta$-regularity is bounded from below by the least common multiple (lcm) of all entries of $A$.
\end{Lemma}

Next, we state two upper bounds, where the first is tighter but the second can be calculated more efficiently:
\begin{Lemma}[Upper Bound Minimal $\Delta$-Regularity]\label{Lemma:upperbound}
	Let $A$ be an integral matrix.
 	\begin{enumerate}[label=(\roman*)]
		\item Let $S$ be the set containing the determinants of all non-singular quadratic submatrices of $A$.
		Then, matrix $A$ is $\lcm\{S\}$-regular.
		\item Let $a_i$ be the rows of $A$ and
		\begin{equation}\label{eq:upperbound}
		\Delta^{UB} \coloneqq \lcm \left\{1,\dots, \left\lfloor \prod_{i=1}^m \max\{1, || a_i ||_2\} \right\rfloor\right\}.
		\end{equation}
		Then, matrix $A$ is $\Delta^{UB}$-regular.
		\label{Lemma:upperbound:maxdet}
	\end{enumerate}
	These $\Delta$-regularities of matrix $A$ are per definition an upper bound for its minimal $\Delta$-regularity.
\end{Lemma}
\begin{proof}
	The first statement follows from the fact that the inverse matrix of each non-singular quadratic submatrix $R$ of $A$ multiplied by $\det(R)$ results in an integer matrix.
	If set $S$ is unknown, a superset can be efficiently determined.
	Let $R$ be a non-singular $(w\times w)$-submatrix of $A$ and $r_i$ its rows with $i=1,\ldots,w$.
	Without loss of generality we assume that we take the first $w$ columns and rows of $A$.
	Using the Hadamard inequality~\cite{Horn_Johnson_1990}, we obtain the following estimate:  
	\begin{equation}
	| \det(R) | \leq \prod_{i=1}^w || r_i ||_2 \leq \prod_{i=1}^w || a_i ||_2 \leq \prod_{i=1}^m \max\{1,|| a_i ||_2\}. \nonumber
	\end{equation}
	So, all submatrices have a determinant in $\left\{1,\dots, \left\lfloor \prod_{i=1}^m \max\{1, || a_i ||_2\} \right\rfloor\right\}$.
	Rounding down is possible since matrix $A$ is integral.
\end{proof}
An example should illustrate Lemma~\ref{Lemma:lowerbound} and Lemma~\ref{Lemma:upperbound}.
	Consider the two matrices
	\begin{equation}
	A_3=
	\begin{pmatrix}
	1 & 1 \\ -1 & 4
	\end{pmatrix} \text{\quad and\quad}
	A_4=
	\begin{pmatrix}
	1 & 1 \\ 0 & 2
	\end{pmatrix}.\nonumber
	\end{equation}
	A lower bound for the minimal $\Delta$-regularity of $A_3$ is given by $\lcm\{1,4\}=4$.
	A tight upper bound is given with $S=\{1,4,5\}$ by $\lcm\{1,4,5\} = 20$
	and another upper bound can be determined by $\lcm\{1,\ldots,\lfloor \sqrt{2}\cdot\sqrt{17} \rfloor\}=\lcm\{1,2,3,4,5\} = 60$.
    For $A_4$, the lower and both upper bounds are the same as $S=\{1,2\}$ and $\lcm\{1,2\}=\lcm\{1,\ldots,\lfloor\sqrt{2}\cdot\sqrt{4}\rfloor\}=2$.

Let us make an addition to Lemma~\ref{Lemma:upperbound}~\ref{Lemma:upperbound:maxdet}.
\begin{Remark}\label{Remark:upperbound}
  If matrix $A\in\mathbb{R}^{m\times n}$ is not square,
  an alternative estimation for the maximal determinant of a $(w \times w)$-submatrix $R$ of $A$ in Lemma~\ref{Lemma:upperbound} is potentially stronger.
  Since matrix transposition preserves $\Delta$-regularity,
  let w.l.o.g $n < m$. Then
  \begin{equation*}
    | \det(R) | \leq \left(\max_{i=1,\ldots,w}\{|| a_i ||_2 \}\right)^{n}
  \end{equation*}
  applies.
  This is probably of special interest if $n$ is much smaller than $m$.
\end{Remark}
We provide an example to illustrate Remark~\ref{Remark:upperbound}.
Consider the matrix
\[
A_5 = \begin{pmatrix}
  -1 & -1 \\ 0 & -2 \\ 0 & 2 \\ -1 & 1 \\ 1 & 0 \\ 1 & 1 \end{pmatrix}
\]
with $m=6$ and $n=2$, which is minimal $2$-regular.
Using Lemma~\ref{Lemma:upperbound}, we obtain the upper bound $\lcm=\{1,\ldots,\lfloor 8\cdot\sqrt{2} \rfloor\}=27720$
and with Remark~\ref{Remark:upperbound} we get the upper bound $\lcm=\{1,\ldots,4\}=12$.

From Lemma~\ref{Lemma:upperbound} follows that we can state
for an arbitrary matrix $A\in \mathbb{Z}^{m\times n}$ a valid $\Delta$-regularity in polynomial time.
It should be noted, however, that this upper bound can be large, making its usefulness questionable in certain situations.

In the following, we will determine a minimal $\Delta$-regularity for three different models.
First we consider a model for the capacitated single-item lot-sizing problem (CLS) (see e.g.~\cite{Pochet2006}).
The variables $s_t$, $x_t$, and $y_t$ represent the inventory at the end of time period $t$,
the production lot-size of period $t$, and an indication if production occurs in period $t$.
The data $h_t$, $p_t$, $q_t$, $d_t$, and $c_t$ represent the unit inventory cost in period $t$, the unit production cost in period $t$,
the fixed cost in period $t$, the demand to be satisfied in period $t$, and the restriction of the production lot-size in period $t$, respectively.
It can be formulated as mixed-integer linear optimization problem as
\begin{subequations}\label{eq:CLS}
  \begin{align}
    \underset{s,x,y}{\min} \quad & \sum_{t=0}^\eta h_t s_t + \sum_{t=1}^\eta p_t x_t + \sum_{t=1}^\eta q_t y_t\\
    \st \quad& s_{t-1} - s_t + x_t = d_t , && \foralltext t=[\eta],\label{eq:CLS:demand}\\
    & x_t \leq c_t y_t, && \foralltext t=[\eta],\label{eq:CLS:capa}\\
    & s \in \mathbb{R}_+^{\eta+1},\, x\in\mathbb{R}_+^{\eta},\, y\in\{0,1\}^{\eta},\label{eq:CLS:variables}\\
    & d\in\mathbb{Z}_+^{\eta},\, c \in \mathbb{Z}_+^{\eta},\label{eq:CLS:integer-data}
  \end{align}
\end{subequations}
where $[\eta] \coloneqq \{1,\dots,\eta\}$. The constraints~\eqref{eq:CLS:demand} are often referred to as flow conservation constraints and 
the constraints~\eqref{eq:CLS:capa} are called capacity constraints.
Constraint~\eqref{eq:CLS:capa} forces the binary variable $y_t$ in period $t$ to be $1$, whenever there is production in period $t$.
In addition to the general modeling of the capacitated single-item lot-sizing problem,
we additionally require the integrality of $d$ and $c$ in \eqref{eq:CLS:integer-data},
in order to get an integer right-hand side by fixing the $y$ variables.
In \cite[Example~6.2]{Wolsey2020} it was shown that the decision problem belonging to problem~\eqref{eq:CLS} is $\mathcal{NP}$-hard.

We want to show how a $\Delta$-regularity can be determined depending on the coefficients of the constraints in problem~\eqref{eq:CLS}.
\begin{theorem}\label{Theorem:CLS:1-regular}
  The matrix belonging to the continuous variables $s$ and $x$ in problem~\eqref{eq:CLS} is $\Delta^{CLS}$-regular,
  where $\Delta^{CLS}=1$.
\end{theorem}
\begin{proof}
  Let $(\bar{s}, \bar{x}, \bar{y})$ be an arbitrary feasible solution for problem~\eqref{eq:CLS}.
  With the matrix
  \begin{equation}\label{eq:CFL:matrix:B}
    H = \begin{pmatrix} 1 & -1 & 0 & 0 & 0 & 0 \\
      0 & 1 & -1 & 0 & 0 & 0 \\
      0 & 0 & \ddots & \ddots & 0 & 0 \\
      0 & 0 & 0 & 1 & -1 & 0  \\ 
      0 & 0 & 0 & 0 & 1 & -1 \end{pmatrix} \in \{-1,0,1\}^{\eta \times (\eta+1)}
  \end{equation}
  the constraints of problem~\eqref{eq:CLS} for fixed $\bar{y}$ can also be written as
  \begin{equation}\label{eq:CLS:FixedMatrixVectorForm}
    \underbrace{\begin{pmatrix} H & I_\eta  \\ -H & -I_\eta  \\ 0 & -I_\eta \end{pmatrix}}_{= A \in \mathbb{Z}^{(3\eta)\times (2\eta+1)}} \begin{pmatrix} s \\ x \end{pmatrix} \geq
    \underbrace{\begin{pmatrix} d \\ -d \\ -\diag(c) \bar{y} \end{pmatrix}}_{=g-B\bar{y}\in\mathbb{Z}^{3\eta}}. 
  \end{equation}
  Due to \cite[Lemma 6]{Appa_Kotnyek_2004} and \cite[Lemma 7]{Appa_Kotnyek_2004} it is sufficient to determine the regularity for submatrix
  $H$ in order to obtain the regularity for matrix $A$.
  Let $h_{ij}$ the entry in row $i$ and column $j$ of the matrix $H$.
  We show three properties of $H$.
  First it holds $h_{ij} \in \{-1,0,1\}$,
  second each column $j\in\{1,\ldots,\eta+1\}$ of matrix $H$ contains at most two nonzero coefficients with $\sum_{i=1}^\eta \vert b_{ij}\vert \leq 2$, and
  third we can partition the rows of $H$ by $M_1 \coloneqq \{1,\ldots,\eta\}$ and $M_2 \coloneqq \emptyset$ such that $\sum_{i \in M_1} b_{ij} - \sum_{i\in M_2}b_{ij} = 0$
  for each column $j$ containing two nonzero coefficients.
  With \cite[Proposition~3.2]{Wolsey2020} follows that $H$ is totally unimodular or in other words $1$-regular.
  This shows that the matrix $A$ is $1$-regular, that is, $\Delta^{CLS}=1$.
\end{proof}

A similar procedure as with Theorem~\ref{Theorem:CLS:1-regular}
for the capacitated single-item lot-sizing problem can also be applied to other problems as for example the
capacitated fixed charge network flow problem (see,~e.g.,~\cite[p.~58]{Wolsey2020}).

Next, we consider the multi-item single-level lot-sizing problem with joint resource constraints (MISL), see \cite[p.~384]{Pochet2006}.
To represent the model, we use the same identifiers as for problem~\eqref{eq:CLS}.
The superscript $i$ denotes the assignment to item $i$.
The coefficients $a^i$ and $b^i$ represent the amount of capacity of the resource $r_t$ consumed per unit
of item $i$ produced in period $t$, and for a set-up of item $i$ in period $t$, respectively.
It can be formulated as mixed-integer linear optimization problem as
\begin{subequations}\label{eq:MISL}
	\begin{align}
	\underset{s,x,y}{\min} \quad & \sum_{i=1}^\mu\sum_{t=0}^\eta h_t^i s_t^i + \sum_{i=1}^\mu \sum_{t=1}^\eta p_t^i x_t^i + \sum_{i=1}^\mu \sum_{t=1}^\eta q_t^i y_t^i \span\span\\
	\st \quad& s_{t-1}^i - s_t^i + x_t^i = d_t^i, && \foralltext i \in [\mu],\,t \in [\eta],\label{eq:MISL:store}\\
	& x_t^i \leq c_t^i y_t^i, && \foralltext i \in [\mu],\,t \in [\eta],\label{eq:MISL:capa}\\
	& \sum_{i=1}^\mu a^i x_t^i + \sum_{i=1}^\mu b^i y_t^i \leq r_t, && \foralltext t\in [\eta],\label{eq:MISL:rescapa}\\
	& s\in\mathbb{R}_+^{\mu\eta+\mu},\, x\in\mathbb{R}_+^{\mu\eta},\, y\in\{0,1\}^{\mu\eta},\label{eq:MISL:variables}\\
        & d \in \mathbb{Z}_+^{\mu\eta},\, c \in \mathbb{Z}_+^{\mu\eta},\, a\in\mathbb{Z}_+^{\mu},\, b\in\mathbb{Z}_+^{\mu},\, r\in\mathbb{Z}_+^\eta.\span\span\label{eq:MISL:integer-data}
	\end{align}
\end{subequations}
The model~\eqref{eq:MISL} shows that $\mu$ capacitated single-item lot-sizing problems~\eqref{eq:CLS}
are coupled via common resource constraints~\eqref{eq:MISL:rescapa}.
For the same reason as for problem~\eqref{eq:CLS}, we require integrality for some input data of the model in \eqref{eq:MISL:integer-data}.
Since problem~\eqref{eq:CLS} is a subproblem of \eqref{eq:MISL},
it follows that the corresponding decision problem to \eqref{eq:MISL} is also a $\mathcal{NP}$-hard problem.

\begin{theorem}\label{Theorem:MISL:Delta-regular}
  The matrix belonging to the continuous variables $s$ and $x$ in problem~\eqref{eq:MISL} is $\Delta^{MISL}$-regular,
  where $\Delta^{MISL}$ is the least common multiple of coefficients $a^i$ for all $i\in[\mu]$.
\end{theorem}
\begin{proof}
  Let $(\bar{s}, \bar{x}, \bar{y})$ be an arbitrary feasible solution for problem~\eqref{eq:MISL}.
  With the matrix $H$ as defined in equation~\eqref{eq:CFL:matrix:B}
  the constraints of problem~\eqref{eq:MISL} for fixed $\bar{y}$ can also be written as
  \begin{equation*}
    \underbrace{\begin{pmatrix} H & 0 & 0 & I_\eta & 0 & 0 \\
        0 & \ddots & 0 & 0 & \ddots & 0 \\
        0 & 0 & H & 0 & 0 & I_\eta \\
        -H & 0 & 0 & -I_\eta & 0 & 0 \\
        0 & \ddots & 0 & 0 & \ddots & 0 \\
        0 & 0 & -H & 0 & 0 & -I_\eta \\
        0 & 0 & 0 & -I_\eta & 0 & 0 & \\
        0 & 0 & 0 & 0 & \ddots & 0 \\
        0 & 0 & 0 & 0 & 0 & -I_\eta \\
        0 & 0 & 0 & -a^1 I_\eta & \cdots  & -a^\mu I_\eta \\
    \end{pmatrix}}_{= A \in \mathbb{Z}^{(3\mu\eta+\eta)\times (2\mu\eta+\mu)}} \begin{pmatrix} s^1 \\ \vdots \\ s^\mu \\ x^1 \\ \vdots \\ x^\mu \end{pmatrix} \geq
    \underbrace{\begin{pmatrix} d^1 \\ \vdots \\ d^\mu \\ -d^1 \\ \vdots \\ -d^\mu \\ -\diag(c^1) (\bar{y}^{1}) \\ \vdots \\ -\diag(c^\mu) (\bar{y}^{\mu}) \\ \sum_{i=1}^\mu(b^i I_\eta (\bar{y}^i))-r \\
    \end{pmatrix}}_{=g-B\bar{y}\in\mathbb{Z}^{3\mu\eta+\eta}}. 
  \end{equation*}
  Due to \cite[Lemma 6]{Appa_Kotnyek_2004} and \cite[Lemma 7]{Appa_Kotnyek_2004} it is sufficient to determine the regularity for submatrix
  \begin{equation}
    M \coloneqq \left(
    \begin{matrix}
      H & 0 & 0	& I_\eta & 0 & 0 \\
      0	& \ddots & 0 & 0 & \ddots & 0\\
      0 & 0 & H & 0 & 0 & I_\eta \\[5pt]
      0 & 0 & 0 & a^1 I_\eta & \cdots & a^\mu I_\eta \\
    \end{matrix} \right)
  \end{equation}
  of $A$ in order to obtain the regularity for matrix $A$.
  Note that we get the matrix $M$ by taking the rows 1-3 and 10 from the matrix $A$.

  If all coefficients $a^i$ are 1, matrix $M$ is totally unimodular~\cite[Proposition~3.2]{Wolsey2020}
  and thus matrix $M$ is 1-regular.
  Due to \cite[Lemma 14]{Appa_Kotnyek_2004}, this matrix multiplied with $\lcm(a^1,\dots,a^\mu)$ is $\lcm(a^1,\dots,a^\mu)$-regular.
  Due to \cite[Lemma 6]{Appa_Kotnyek_2004} dividing a row or column by a nonzero integer preserves the regularity.
  So, for each item, dividing all columns corresponding to this item by the integer value $\frac{\lcm(a^1,\dots,a^\mu)}{a^i}$,
  and then dividing all rows corresponding to constraints~\eqref{eq:MISL:store} of this item by $a^i$
  leads to the original matrix $M$ and preserves the $\lcm(a^1,\dots,a^\mu)$-regularity.
  Please note that the assumption of $a^i \not= 0$ for all items $i$ makes sense
  in order to avoid division by $0$, because otherwise the model would be trivially decomposable.
  So, $\Delta^{MISL} = \lcm(a^1,\dots,a^\mu)$ holds.\qedhere
\end{proof}

As last example for determining $\Delta$-regularity,
we consider the capacitated facility location problem (CFL) (see e.g.~\cite{Wolsey2020}).
The capacitated facility location problem is to find locations for new facilities
such that the conveying cost from facilities to customers is minimized.
Specifically there are $\eta$ potential facilities and $\mu$ clients.
Facility $j$ has a capacitiy $r_j$ and client $i$ has a demand $a^i$.
There is a fixed cost $c_j$ of using facility $j$ and $h^i_j$ is the cost of transporting $a^i$ units from facility $j$ to client $i$.
The goal is to minimize the total cost while satisfying the demands subject to the capacity constraints.
Letting $y_j=1$ if facility $j$ is opened and $0$ otherwise and $x^i_j$ denote the fraction
of the demand of client $i$ satisfied from facility $j$.
It can be formulated as mixed-integer linear optimization problem as
\begin{subequations}\label{eq:CFL}
	\begin{align}
	\underset{x,y}{\min} \quad & \sum_{j=1}^\eta c_j y_j + \sum_{i=1}^\mu \sum_{j=1}^\eta h^i_j x^i_j\\
	\st \quad& \sum_{j=1}^\eta x^i_j = 1, && \foralltext i=[\mu],\label{eq:CFL:demand}\\
	& \sum_{i=1}^\mu a^i x^i_j \leq r_j y_j, && \foralltext j=[\eta],\label{eq:CFL:capa}\\
        & \sum_{j=1}^\eta r_j y_j \geq \sum_{i=1}^\mu a^i, \label{eq:CFL:globalCut}\\
	& x\in\mathbb{R}_+^{\mu\eta},\, y\in\{0,1\}^{\eta},\label{eq:CFL:variables}\\
        & a \in \mathbb{Z}_+^{\mu},\, r \in \mathbb{Z}_+^{\eta}.\label{eq:CFL:integer-data}
	\end{align}
\end{subequations}
The constraints~\eqref{eq:CFL:demand} ensure that the demand of each client is met and
the constraints~\eqref{eq:CFL:capa} ensure that the capacity of facility $j$ is not exceeded.
The inequality \eqref{eq:CFL:globalCut} enforces that sufficient capacity is available to meet the entire demand.
For the same reason as for problem~\eqref{eq:CLS}, we require integrality for some input data of the model in \eqref{eq:CFL:integer-data}.
It is known that the decision problem for \eqref{eq:CFL} is $\mathcal{NP}$-hard~\cite{Buesing_etal_2022}.

\begin{theorem}\label{Theorem:CFL:Delta-regular}
  The matrix belonging to the continuous variables $x$ in problem~\eqref{eq:CFL} is $\Delta^{CFL}$-regular,
  where $\Delta^{CFL}$ is the least common multiple of coefficients $a^i$ for all $i\in[\mu]$.
\end{theorem}

\begin{proof}
	This can be proven with the same arguments as for Theorem~\ref{Theorem:MISL:Delta-regular}.
\end{proof}

An extension of the capacitated facility location problem is to require single-sourcing for a subset of clients~\cite{WeningerWolsey2023},
that is, there are clients that may only be supplied by one facility.
For this problem extension, however, $\Delta^{CFL}$ from Theorem~\ref{Theorem:CFL:Delta-regular} remains valid because the matrix to be considered is a submatrix~\cite[Lemma 6]{Appa_Kotnyek_2004}.

Following from Lemma~\ref{Lemma:lowerbound}, which gives a lower bound, and the Theorems~\ref{Theorem:MISL:Delta-regular} and~\ref{Theorem:CFL:Delta-regular}, which give upper bounds, we can state:
\begin{Corollary}
	$\Delta^{MISL}$ and $\Delta^{CFL}$ are minimal $\Delta$-regularities.
\end{Corollary}

%% file: section_decbranching.tex
\section{Enhancement of Decomposition Branching}
\label{sect:decbranching}

We apply our ideas on the \emph{Decomposition Branching} (DB) approach, introduced in~\cite{Yildiz_et_al:2022}.
The branching scheme of Decomposition Branching is quite different to branching schemes based on single-variable dichotomy as
it divides the search space by exploiting \emph{decomposition information} (or for short \emph{decomposition}) of the problem.

A decomposition with $k$ blocks for a matrix is a partition of the rows and columns into $k + 1$ pieces each such that rearranging the matrix according to the decomposition yields a block structure with $k$ blocks on the diagonal and so-called \emph{linking rows} and/or \emph{linking columns}.
For our application we only consider linking rows and no linking columns, thus matrix $A$ can be rewritten as
\begin{align*}
A =
	\left(
	\begin{matrix}
		A_{1} &	0 &	0\\
		0 &	\ddots & 0 \\
		0 & 0 & A_{k} \\
		{A}^{\text{row}}_1 & \cdots & {A}^{\text{row}}_k
	\end{matrix}
	\right).
\end{align*}
For each block $q \in [k]$,
the matrix $A_q$ is a submatrix of $A$ containing exactly the rows and columns of block $q$.
The last rows are the linking rows, where $A_q^\text{row}$, $q\in[k]$, are the submatrices containing exactly the columns of block $q$ and only linking rows.
Restrictions of the matrix $B$ and the vectors $c$, $d$, $g$, $x$, and $y$ of problem~\eqref{eq:MIP} are defined analogously.
A more detailed description of decompositions can be found, for example, in~\cite{Halbig_DecHeur_2023}.

\subsection{Properties of the initial Algorithm}

At first, we present the basic concept of Decomposition Branching following the presentation in~\cite{Yildiz_et_al:2022}.
Assuming, matrix $\begin{pmatrix}A & B \end{pmatrix}$ of problem~\eqref{eq:MIP} has a block structure with $k$ blocks and $m^\text{row}$ linking rows.
It is then possible to rewrite problem~\eqref{eq:MIP} as 
\begin{subequations}\label{eq:DB:orig}
	\begin{align}
	\underset{x,y}{\min} \quad & \sum_{q\in[k]}c_q^\top x_q + d_q^\top y_q\\
	\st \quad& A_q x_q + B_q y_q \geq g_q, &&\foralltext q \in [k],\\
	& \sum_{q\in[k]}A_q^\text{row}x_q + \sum_{q\in[k]}B_q^\text{row}y_q \geq g^\text{row}, \label{eq:DB:orig:linking}\\
	& x \geq 0,\ y\geq 0,\ y \in \mathbb{Z}^{\ell}.
	\end{align}
\end{subequations}

The so-called \emph{linking constraints}~\eqref{eq:DB:orig:linking} can be further divided into the parts of the single blocks by means of additional variables $p_q \in \mathbb{R}^{m^\text{row}}$ for each block $q \in [k]$,
which describe how the right-hand side vector $g^{\text{row}}$ is partitioned between the blocks.
Such a reformulation with auxiliary variables is also utilized, for example, in~\cite{BodurABN22}.
Thus, problem~\eqref{eq:DB:orig} can be reformulated as
\begin{subequations}\label{eq:DB:reform}
	\begin{align}
	\underset{x,y,p}{\min} \quad & \sum_{q\in[k]}c_q^\top x_q + d_q^\top y_q\\
	\st \quad& A_q x_q + B_q y_q \geq g_q, &&\foralltext q \in [k],\\
	& A_q^{\text{row}}x_q + B_q^{\text{row}}y_q \geq p_q, && \foralltext q \in [k], \\
	& \sum_{q\in[k]} p_q \geq g^{\text{row}},\\
	& x \geq 0,\ y\geq 0,\ y \in \mathbb{Z}^{\ell}.
	\end{align}
\end{subequations}

Decomposition Branching solves problem~\eqref{eq:DB:orig} by a branch-and-bound approach using the following branching rule.
At some node of the branch-and-bound tree the LP relaxation of problem~\eqref{eq:DB:orig} is solved obtaining solution $(x^{\text{LP}},y^{\text{LP}})$.
If $y^{\text{LP}}$ is integral, this solution is optimal for the current node and feasible for~\eqref{eq:DB:orig}, and the node can be pruned by optimality.
Otherwise, for each block $q\in[k]$ with noninteger $y^{\text{LP}}_q$
the \emph{branching subproblem}
\begin{subequations}\label{eq:DB:subproblem}
\begin{align}
z^*_q(x^{\text{LP}}_q,y^{\text{LP}}_q) &=& \underset{x,y}{\min} \quad & c_q^\top x_q + d_q^\top y_q\\
&& \st \quad & A_q x_q + B_q y_q \geq g_q,\\
&&& A_q^{\text{row}}x_q + B_q^{\text{row}}y_q \geq A_q^{\text{row}}x^{\text{LP}}_q + B_q^{\text{row}}y^{\text{LP}}_q, \\
&&& x_q \geq 0,\ y_q \geq 0,\ y_q \in \mathbb{Z}^{\ell_q},
\end{align}
\end{subequations}
is solved to obtain solution $(x^*_q,y^*_q)$.
Problem~\eqref{eq:DB:subproblem} arises from~\eqref{eq:DB:reform}, whereby only variables of block $q$ are considered and variables $p_q$ are fixed at $A_q^{\text{row}}x^{\text{LP}}_q + B_q^{\text{row}}y^{\text{LP}}_q$.
For each block $q\in[k]$ with integer LP solution $y^{\text{LP}}_q$
the LP solution is the optimal solution of the branching subproblem,
thus set $z^*_q(x^{\text{LP}}_q,y^{\text{LP}}_q) = c_q^\top x^{\text{LP}}_q + d_q^\top y^{\text{LP}}_q$
and $(x^*_q,y^*_q) = (x^{\text{LP}}_q,y^{\text{LP}}_q)$.
If it turns out that for every block
$z^*_q(x^{\text{LP}}_q,y^{\text{LP}}_q) = c_q^\top x^{\text{LP}}_q + d_q^\top y^{\text{LP}}_q$
holds, then it must be that
$(x^*_1,\dots,x^*_k, y^*_1,\dots,y^*_k)$
is an optimal solution for the current node of the branch-and-bound tree and the node is pruned by optimality.

If the node is not pruned, one stops at the first block $q$ for which the branching subproblem has not the same solution value as the LP solution,
that is, $z^*_q(x^{\text{LP}}_q,y^{\text{LP}}_q) > c_q^\top x^{\text{LP}}_q + d_q^\top y^{\text{LP}}_q$.
Suppose $q$ is the first such branching subproblem, the following branching rule is proposed:
\begin{equation}\label{eq:branchingrule:orig}
\begin{aligned}
&c_q^\top x_q + d_q^\top y_q \geq z^*_q(x^{\text{LP}}_q,y^{\text{LP}}_q)\\
&\vee \bigvee_{j=1}^{\ \ \ m^\text{row}} A_q^{\text{row},j}x_q + B_q^{\text{row},j}y_q < A_q^{\text{row},j}x^{\text{LP}}_q + B_q^{\text{row},j}y^{\text{LP}}_q,
\end{aligned}
\end{equation}
with $A_q^{\text{row},j}$ (or $B_q^{\text{row},j}$) the $j$th row of $A_q^{\text{row}}$ (or $B_q^{\text{row}}$).
Disjunction~\eqref{eq:branchingrule:orig} generates a multi-way branch with $m^\text{row} + 1$ child nodes that cuts off the LP solution $(x^{\text{LP}},y^{\text{LP}})$.

Since solvers usually can not handle strict inequalities,
the authors of~\cite{Yildiz_et_al:2022} suggest to make use of an $\varepsilon$-reformulation.
Thus, a small tolerance parameter $\varepsilon > 0$ is chosen,
and each strict inequality of~\eqref{eq:branchingrule:orig} gets replaced by
\begin{equation}\label{eq:epsreform}
A_q^{\text{row},j}x_q + B_q^{\text{row},j}y_q \leq A_q^{\text{row},j}x^{\text{LP}}_q + B_q^{\text{row},j}y^{\text{LP}}_q - \varepsilon.
\end{equation}
If $A_q^{\text{row},j}x_q + B_q^{\text{row},j}y_q$ is known to be integer in any feasible solution, they suggest to make in addition use of integer rounding and replace the strict inequality by $A_q^{\text{row},j}x_q + B_q^{\text{row},j}y_q \leq \lfloor A_q^{\text{row},j}x^{\text{LP}}_q + B_q^{\text{row},j}y^{\text{LP}}_q - \varepsilon \rfloor$.
In general, this rounding step is not possible, when the linking constraints~\eqref{eq:DB:orig:linking} contain continuous variables.

The special challenges of continuous or mixed-integer linking constraints and the general procedure of DB is illustrated using an example.
Consider the problem
	\begin{align}\label{eq:example:epsilon}
		\underset{x,y}{\min} \quad& y_1 + y_2\\
		\st \quad
		&\begin{pmatrix} 3 & 0\\
			-3 & 0\\
			-1 & 0 \\
			\cmidrule(rl){1-2}
			0& 1 \\
			0& -3 \\
			0& 0 \\
			0& 0 \\
			\cmidrule(rl){1-2}
			1 & 1\end{pmatrix}
			\begin{pmatrix}
				x_1\\
				x_2
			\end{pmatrix}+
		\begin{pmatrix} 1 & 0\\
			-2 & 0\\
			3 & 0\\
			\cmidrule(rl){1-2}
			0& 0 \\
			0& 1 \\
			0& 1 \\
			0& -1 \\
			\cmidrule(rl){1-2}
			0& 0
		\end{pmatrix}
		\begin{pmatrix}
		y_1\\
		y_2
		\end{pmatrix}
		\geq 
		\begin{pmatrix} 2 \\
			-3 \\
			0 \\
			\cmidrule(rl){1-1}
		    0 \\
	    	-2\\
    		0\\
    		-1\\
    		\cmidrule(rl){1-1}
    		1
    	\end{pmatrix},\qquad
    	\begin{matrix} \\
    	\text{(block 1)} \\
    	 \\
    	 \cmidrule(r){1-1}
    	\\
    	\text{(block 2)} \\
    	\\
    	\\
    	\cmidrule(r){1-1}
    	\text{(linking cons.)\hspace*{-2pt}}\nonumber
    	\end{matrix}\\
		& x \geq 0,\ y\geq 0,\ y \in \mathbb{Z}^2,\nonumber
	\end{align}
with two blocks and one linking constraint.
This example is illustrated in Figure~\ref{fig:example:epsilon}.

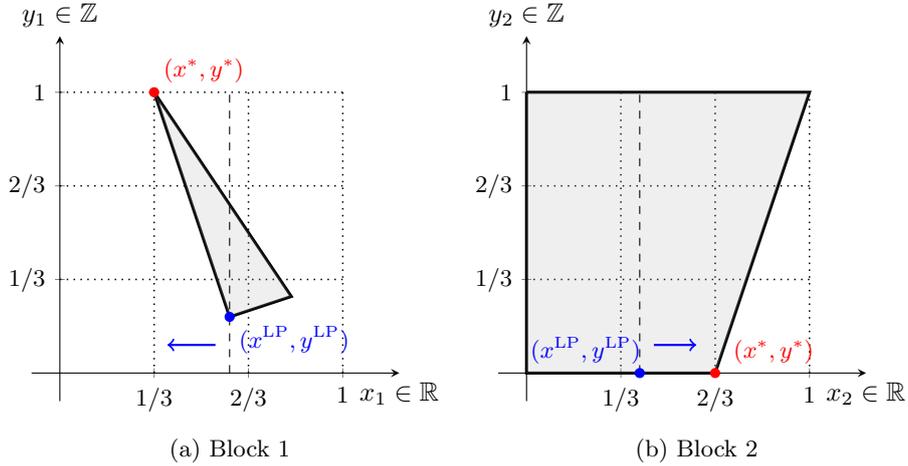
\begin{figure}[ht]
	\captionsetup[subfigure]{justification=centering}
	\captionsetup[subfigure]{format=hang}
	\definecolor{colorrootplot}{rgb}{0.12,0.46,0.70}
	\pgfplotsset{
		block plot/.style = {
			ticklabel style = {font=\small},
			unit vector ratio={1 1},
			axis x line=center,
			axis y line=center,
			xtick={1/3,2/3,1},
			ytick={1/3,2/3,1},
			xticklabels={1/3,2/3,1},
			yticklabels={1/3,2/3,1},
			xlabel style={below},
			ylabel style={above},
			xmin=-0.1, xmax=1.2,
			ymin=-0.1, ymax=1.2,
			scale=0.85
		},
	}
	\begin{subfigure}[b]{0.5\textwidth}
		\begin{tikzpicture}
		\begin{axis}[block plot,
		xlabel={$x_1\in \mathbb{R}$},
		ylabel={$y_1\in \mathbb{Z}$}
		]
		\addplot[dotted, line width = 0.2mm] (0,1/3) -- (1,1/3);
		\addplot[dotted, line width = 0.2mm] (0,2/3) -- (1,2/3);
		\addplot[dotted, line width = 0.2mm] (0,1) -- (1,1);
		\addplot[dotted, line width = 0.2mm] (1/3,0) -- (1/3,1);
		\addplot[dotted, line width = 0.2mm] (2/3,0) -- (2/3,1);
		\addplot[dotted, line width = 0.2mm] (1,0) -- (1,1);

		\addplot[line width = 0.4mm] (1/3,1) -- (9/11,3/11) -- (0.6,0.2) -- (1/3,1);
		\fill[fill=gray!60, fill opacity=0.2] (1/3,1) -- (9/11,3/11) -- (0.6,0.2) -- (1/3,1);

		\addplot[dashed] coordinates {(0.6, 0) (0.6, 1)};
		\addplot[mark=*, color=blue] coordinates {(0.6, 0.2)} node[below right]{\small $(x^{\text{LP}},y^{\text{LP}})$};
		\addplot[mark=*, color=red] coordinates {(1/3,1)} node[above right]{\small$(x^*,y^*)$};

		\draw[->, color=blue, thick] (0.55,0.1) -- (0.38,0.1);
		\end{axis}
		\end{tikzpicture}
		\caption{Block 1}
		\label{fig:example:epsilon:1}
	\end{subfigure}
	\begin{subfigure}[b]{0.5\textwidth}
		\begin{tikzpicture}
		\begin{axis}[block plot,
		xlabel={$x_2\in \mathbb{R}$},
		ylabel={$y_2\in \mathbb{Z}$}
		]
		\addplot[dotted, line width = 0.2mm] (0,1/3) -- (1,1/3);
		\addplot[dotted, line width = 0.2mm] (0,2/3) -- (1,2/3);
		\addplot[dotted, line width = 0.2mm] (0,1) -- (1,1);
		\addplot[dotted, line width = 0.2mm] (1/3,0) -- (1/3,1);
		\addplot[dotted, line width = 0.2mm] (2/3,0) -- (2/3,1);
		\addplot[dotted, line width = 0.2mm] (1,0) -- (1,1);

		\addplot[line width = 0.4mm] (0,1) -- (1,1) -- (2/3,0) -- (0,0)-- (0,1);
		\fill[fill=gray!60, fill opacity=0.2] (0,1) -- (1,1) -- (2/3,0) -- (0,0)-- (0,1);

		\addplot[dashed] coordinates {(0.4, 0) (0.4, 1)};
		\addplot[mark=*, color=blue] coordinates {(0.4, 0)} node[above left]{\small $(x^{\text{LP}},y^{\text{LP}})$\hspace*{-4pt}};
		\addplot[mark=*, color=red] coordinates {(2/3,0)} node[above right]{\ \small$(x^*,y^*)$};

		\draw[->, color=blue, thick] (0.45,0.1) -- (0.6,0.1);
		\end{axis}
		\end{tikzpicture}
		\caption{Block 2}
		\label{fig:example:epsilon:2}
	\end{subfigure}
	\caption{Example of initial Decomposition Branching.
          For each block, the corresponding polyhedron is shown as a grey shaded area with the LP solution of the first iteration in blue and the optimal solution in red.}
     	\label{fig:example:epsilon}
\end{figure}

We apply Decomposition Branching in its initial form.
By solving the LP relaxation we obtain $(x^{\text{LP}}_1,x^{\text{LP}}_2,y^{\text{LP}}_1,y^{\text{LP}}_2) = (0.6, 0.4, 0.2, 0)$.
Since $y_1^{\text{LP}} \not\in \mathbb{Z}$ we solve the branching subproblem for block 1 and detect its infeasibility.
So we branch without a child node for the objective function leading to only one child node with additional constraint $x_1 < 0.6$.

The strict inequality is reformulated with an $\varepsilon$ and the LP relaxation for the only child node is solved
obtaining solution $(x^{\text{LP}}_1,x^{\text{LP}}_2,y^{\text{LP}}_1,y^{\text{LP}}_2) = (0.6-\varepsilon, 0.4+\varepsilon, 0.2+3\varepsilon, 0)$.
These steps are repeated until $x_1\leq \frac{1}{3}$ is added to the child node.

A closer look at this approach reveals various properties.
We need $\frac{0.6-\frac{1}{3}}{\varepsilon}$ many iterations until we reach the optimal solution of $(\frac{1}{3}, \frac{2}{3}, 1, 0)$.
Even worse, if $\frac{0.6-\frac{1}{3}}{\varepsilon}$ is not integral, the optimal solution is not found.
In iteration $\lfloor \frac{0.6-\frac{1}{3}}{\varepsilon} \rfloor$ the LP solution is close to the optimal solution but not integral.
Just one iteration later, the optimal solution is cut off and DB stops with infeasibility as result due to two different reasons.
In block 1 the only optimal solution is cut off leading to the result that the original problem is classified as infeasible.
Considering only block 2, also the optimal solution is cut off, but not all feasible solutions,
leading to the result that DB could still return a feasible solution for the original problem.

The observed drawbacks of initial DB approach illustrated above are summarized in the following remark:
\begin{Remark}\label{remark:DB:epsilon}
  Two main drawbacks of an $\varepsilon$-reformulation~\eqref{eq:epsreform}
  utilized in Decomposition Branching for strict inequalities of~\eqref{eq:branchingrule:orig} can be identified:
	\begin{enumerate}[label=(\roman*)]
		\item Many iterations may be required until an optimal solution is reached, especially if tolerance parameter $\varepsilon$ is small.
		\item All optimal solution can get cut off leading to the result that the original problem is classified as infeasible or that a suboptimal solution is returned.
	\end{enumerate}
\end{Remark}

\subsection{Exactness through Exploitation of $\Delta$-Regularity}
\label{sect:decbranching:delta}

The initial Decomposition Branching algorithm can get improved by utilizing $\Delta$-regularity.
At first, one has to determine a suitable $\Delta$ for the regularity of matrix $A$ corresponding to the continuous variables in problem~\eqref{eq:MIP}.
Examples for the determination of a minimal $\Delta$-regularity are given in the previous Section~\ref{sect:misl}.
Subsequently, Decomposition Branching is applied.

Note first, that
for every feasible solution
\begin{equation}
(\bar{x},\bar{y}) \in X \coloneqq \{ (x,y) \,\vert\, Ax + By \geq g,\, x \geq 0,\, y \geq 0,\, y \in \mathbb{Z}^{\ell}\}
\end{equation}
of problem~\eqref{eq:MIP} one can construct a feasible solution
\[
(x^\Delta,\bar{y}) \in X \cap \tfrac{1}{\Delta} \mathbb{Z}^{n+\ell}
\]
with $c^\top x^\Delta + d^\top \bar{y} \leq c^\top \bar{x} + d^\top \bar{y}$.
To prove this, we proceed as follows.
We fix problem~\eqref{eq:MIP} at $\bar{y}$ and solve it.
The optimal solution $x^\Delta$ is a vertex of the underlying polyhedron (see Section~\ref{sect:lattices}) and
thus holds $x^\Delta \in \tfrac{1}{\Delta} \mathbb{Z}^{n}$.
Since $\bar{y}$ is integral, it holds that $(x^\Delta,\bar{y}) \in X \cap \tfrac{1}{\Delta} \mathbb{Z}^{n+\ell}$.
As a direct consequence of this argument, we can state the following lemma.
\begin{Lemma}\label{lemma:optsol}
	For problem~\eqref{eq:MIP} with matrix $A$ being $\Delta$-regular there exists at least one optimal solution
	$(x^{\Delta*},y^{\Delta*}) \in X \cap \tfrac{1}{\Delta} \mathbb{Z}^{n+\ell}$.
\end{Lemma}
Note that there can be further optimal solutions that are not element of lattice $\tfrac{1}{\Delta} \mathbb{Z}^{n+\ell}$.
However, we are only interested in one optimal solution and so it is sufficient to restrict the search.

When applying Decomposition Branching, strict inequalities appear when adding a new child node for linking constraint~$j$ with branching rule~\eqref{eq:branchingrule:orig}.
We give now two statements about strengthening inequalities and converting strict inequalities into inequalities in the next lemma.
\begin{Lemma}\label{lemma:rounding:ineq}
	Given a point $(x^{\Delta},y^{\Delta}) \in  \tfrac{1}{\Delta} \mathbb{Z}^{n+\ell}$,
	two integer vectors $u \in \mathbb{Z}^{n}$ and $w \in \mathbb{Z}^{\ell}$, and scalar $\gamma \in \mathbb{Q}$.
	Then the following statements hold:
	\begin{enumerate}[label=(\roman*)]
		\item If the inequality
		\begin{equation*}
		u^\top x + w^\top y \geq \gamma
		\end{equation*}
		is valid for $(x^{\Delta},y^{\Delta})$,
		then the strengthened inequality
		\begin{equation}\label{eq:ineq-tightened}
		u^\top x + w^\top y \geq \frac{\lceil \Delta \gamma \rceil}{\Delta}
		\end{equation}
		is also valid for $(x^{\Delta},y^{\Delta})$.
		\item If the strict inequality
		\begin{equation*}
		u^\top x + w^\top y > \gamma
		\end{equation*}
		is valid for $(x^{\Delta},y^{\Delta})$,
		then the strengthened inequality
		\begin{equation}\label{eq:strict-ineq:epsilon}
		u^\top x + w^\top y \geq \frac{\lfloor \Delta \gamma \rfloor + 1}{\Delta}
		\end{equation}
		is also valid for $(x^{\Delta},y^{\Delta})$.
	\end{enumerate}
\end{Lemma}
\begin{proof}
	Both statements (i) and (ii) follow from $u^\top x^{\Delta} + w^\top y^\Delta \in  \tfrac{1}{\Delta} \mathbb{Z}^{n+\ell}$.
\end{proof}


The $j$th strict inequality of branching rule~\eqref{eq:branchingrule:orig} can now get replaced by a rounded inequality specified by Lemma~\ref{lemma:rounding:ineq} instead of reformulating it with a small $\varepsilon>0$.
Considering the $j$th strict inequality of~\eqref{eq:branchingrule:orig}, we multiply it with $-1$ and apply~\eqref{eq:strict-ineq:epsilon} to get
\begin{equation}\label{eq:branchingrule:DB:turned}
-A_q^{\text{row},j}x_q - B_q^{\text{row},j}y_q \geq \frac{\lfloor \Delta(-A_q^{\text{row},j}x^{\text{LP}}_q - B_q^{\text{row},j}y^{\text{LP}}_q) \rfloor + 1}{\Delta}.
\end{equation}
Multiplying again with $-1$ and utilizing that $-\lceil f \rceil=\lfloor -f \rfloor$ for a number $f\in\mathbb{R}$ the original strict inequality of~\eqref{eq:branchingrule:orig} can get replaced by
\begin{equation}\label{eq:branchingrule:DB}
A_q^{\text{row},j}x_q + B_q^{\text{row},j}y_q \leq \frac{\lceil \Delta(A_q^{\text{row},j}x^{\text{LP}}_q + B_q^{\text{row},j}y^{\text{LP}}_q) \rceil - 1}{\Delta}.
\end{equation}

Also the nonstrict inequality of~\eqref{eq:branchingrule:orig} corresponding to the objective function can get strengthened by using $\Delta$.
It is sufficient if at least one optimal solution remains in the solution space and thus, utilizing \eqref{eq:ineq-tightened}, the inequality can get replaced by
\begin{equation}\label{eq:branchingrule:DB:obj}
c_q^\top x_q + d_q^\top y_q \geq \frac{\lceil \Delta( z^*_q(x^{\text{LP}}_q,y^{\text{LP}}_q)) \rceil}{\Delta}
\end{equation}
since $c$ and $d$ are integral.

Consider again the example mixed-integer linear problem~\eqref{eq:example:epsilon}.
The matrix $A\in \mathbb{Z}^{8\times 2}$ corresponding to the continuous variables
is 3-regular, thus we know that an optimal solution of this problem is in $\Lambda_3 \coloneqq  \tfrac{1}{3} \mathbb{Z}^{2+2}$.
In the first child node, we have the additional strict inequality $x_1 < 0.6$.
With Lemma~\ref{lemma:rounding:ineq} we can round down the right-hand side to
$\frac{\lceil3\cdot0.6\rceil - 1}{3} = \frac{1}{3}$.
Thus the strict inequality is replaced by $x_1 \leq \frac{1}{3}$.
The LP solution of this child node is $(\frac{1}{3}, 1)$. Since it is integer in $y$, the algorithm stops after one iteration with the optimal solution.

As one can observe, the drawbacks of Remark~\ref{remark:DB:epsilon} are prevented in the above example when utilizing $\Delta$-regularity.

\subsection{Proof of Exactness and Termination}
\label{sect:decbranching:proof}

In the following we proof that the enhanced Decomposition Branching approach utilizing $\Delta$-regularity is exact and terminates after finite many steps.
We call an algorithm \emph{exact} if it returns an optimal solution if available.
This also means, that no (feasibility) tolerances are permitted.
To achieve this, we make the following assumptions:
\begin{assumption}\label{assumption:DBproof}
	We assume that
	\begin{enumerate}[label=(\roman*)]
		\item all subproblems are solved exact,\label{assumption:DBproof:sub}
		\item all arithmetic operations are exact,\label{assumption:DBproof:exact}
		\item and the feasible set of the LP relaxation is a polytope.\label{assumption:DBproof:bounded}
	\end{enumerate}
\end{assumption}

These assumptions are not very strong, yet they are indispensable for theoretical inquiries.
We need the Assumptions \ref{assumption:DBproof} \ref{assumption:DBproof:sub} and \ref{assumption:DBproof} \ref{assumption:DBproof:exact} to obtain an exact optimal solution.
Assumption \ref{assumption:DBproof} \ref{assumption:DBproof:bounded} is needed for the termination.

Now we come to the main result of this section.
\begin{theorem}\label{Theorem:DBterminatesexact}
  Under the Assumptions~\ref{assumption:DBproof} Decomposition Branching utilizing $\Delta$-regularity applied to problem~\eqref{eq:MIP} is exact and terminates after finite many iterations.
\end{theorem}
\begin{proof}
	This theorem is proven in four steps:\\
	First, no optimal solution $(x^{\Delta*},y^{\Delta*}) \in \Lambda_\Delta \coloneqq  \tfrac{1}{\Delta} \mathbb{Z}^{n+\ell}$ is cut off:
	We know that there is at least one optimal solution $(x^{\Delta*},y^{\Delta*})$ element of lattice $\Lambda_\Delta$, see Lemma~\ref{lemma:optsol}.
	Decomposition Branching with the initial branching rule~\eqref{eq:branchingrule:orig} does not cut off integer feasible solutions $(\bar{x},\bar{y}) \in X$.
	Due to Lemma~\ref{lemma:rounding:ineq} we can strengthen the inequalities of branching rule~\eqref{eq:branchingrule:orig} without cutting off points in lattice $\Lambda_\Delta$.
	So in each branching step $(x^{\Delta*},y^{\Delta*}) \in \Lambda_\Delta$ is kept in one of the child nodes.
	
	Second, the algorithm terminates after a finite number of steps:
	The number of different inequalities of~\eqref{eq:branchingrule:orig} is finite since their right-hand sides can take only finite many different values.
	That holds because the right-hand sides are element of $\frac{1}{\Delta}\mathbb{Z}$ and because of Assumption~\ref{assumption:DBproof}~\ref{assumption:DBproof:bounded}.
	Moreover, an inequality can not get added twice.
	Once an inequality is added, the LP solutions of the subsequent child nodes have to satisfy it.
	Consequently, in inequality~\eqref{eq:branchingrule:DB}, the right-hand side for any subsequent LP solution will be strictly less than the right-hand side of the currently added inequality.
	Similarly, the right-hand side in inequality~\eqref{eq:branchingrule:DB:obj} will be strictly greater than the current one.
	Thus, the number of nodes is finite and the algorithm terminates.
     
	Third, the results of the nodes are correct:
	At each node it is either branched or the node is pruned by (i) bound, (ii) infeasibility, (iii) optimality.
	The node is pruned by bound if the LP solution value is worse than the current incumbent.
	This node can not contain optimal solution $(x^{\Delta*},y^{\Delta*})$ and does not need to be considered further.
	A node is pruned by infeasibility in two cases.
	The first case occurs if the LP relaxation is infeasible.
	The second case occurs if the parent node branched on an infeasible branching subproblem.
	Then the child node corresponding to the objective function can get immediately pruned.
	A node is pruned by optimality if the optimal LP solution is integer feasible or if the optimal solution values of all branching subproblems are equal to the optimal LP solution values.
	Note that this solution is not necessarily in lattice $\Lambda_\Delta$.
	But following the explanation in Section~\ref{sect:decbranching:delta}, one can construct a solution in $\Lambda_\Delta$.
	However, since $(x^{\Delta*},y^{\Delta*})$ is never cut off, the best of all found solutions is not worse than $(x^{\Delta*},y^{\Delta*})$.

	Fourth, feasible solutions are exact:
	Due to Assumption~\ref{assumption:DBproof}~\ref{assumption:DBproof:sub} and~\ref{assumption:DBproof:exact} all found integer feasible solutions are exact, that is, feasible without tolerances.
	
	Thus, Decomposition Branching utilizing $\Delta$-regularity terminates after a finite number of steps with an exact optimal solution.
\end{proof}

Note that problem~\eqref{eq:MIP} has an optimal vertex solution per assumption.
Nevertheless, Theorem~\ref{Theorem:DBterminatesexact} still holds when problem~\eqref{eq:MIP} is infeasible.
The Algorithm terminates after finite many iterations with the same argumentation as in the proof above.
Moreover, the problem is infeasible if and only if Decomposition Branching utilizing $\Delta$-regularity terminates without any feasible solution.

%% file: section_results.tex
\section{Computational Study}
\label{sect:results}

In this section, we present a comprehensive computational study to investigate the behavior of the initial Decomposition Branching compared to Decomposition Branching utilizing $\Delta$-regularity.

\subsection{Computational Setup}
Both algorithm variants, the initial Decomposition Branching (DB) and the enhanced version using $\Delta$-regularity ($\Delta$DB), are implemented in C++.
The source code is available at
\cite{githubrepo_decbranch}.
The difference between DB and $\Delta$DB consists solely in the handling of strict inequalities and in strengthening nonstrict inequalities.
DB uses an $\varepsilon$-reformulation where tolerance parameter $\varepsilon$ can be chosen variously but is fixed.
Whereas, our enhanced proposed version, $\Delta$DB,
rounds the right-hand side depending on $\Delta$-regularity of matrix $A$
as shown in Lemma~\ref{lemma:rounding:ineq} and equations~\eqref{eq:branchingrule:DB} and~\eqref{eq:branchingrule:DB:obj}.

We used a development version of \emph{exact SCIP}~\cite{Eifler2023} 
based on a pre-release version of SCIP~10.0
as solver for the subproblems, whose code is publicly available at ~\url{https://github.com/scipopt/scip/tree/exact-rational}.
Thereby, we used SoPlex~7.0~\cite{bolusani2024} as underlying LP solver with MPFR and GMP linked.
Exact SCIP~\cite{Eifler2023} is a substantial revision and extension of~\cite{Cook_etal_2013}.
It is a numerically exact variant of solver SCIP, and also our own code uses exact arithmetic operations.
Thus, Assumptions~\ref{assumption:DBproof}~\ref{assumption:DBproof:sub} and \ref{assumption:DBproof}~\ref{assumption:DBproof:exact} hold, and inexactness can only stem from our algorithm variants themselves.

It should be noted that our code can not compete with state-of-the-art solvers.
One possible explanation for this is the lack of additional methods, such as presolving, heuristics, cut generators, as well as the absence of sophisticated strategies for selecting the next node or branching subproblem in our code.
Furthermore, exact SCIP is approximately 8.1 times slower than the floating-point version of SCIP~\cite{Eifler2023}.

All presented computational results were generated on a compute cluster
using compute nodes with Intel Xeon Gold 6326 processors with 2.9 GHz and
32 GB RAM; see~\cite{WoodyClusterWebsite} for more details.

\subsection{Test Sets}

We consider two sets of artificially generated instances.
The first set contains multi-item single-level lot-sizing problems with joint resource constraints (\testset{MISL}), whereas the second set contains capacitated facility location problems with single-sourcing (\testset{CFL}).
Both underlying models are described in Section~\ref{sect:misl}.

The instances of test set \testset{MISL} have 2, 3, or 4 items, whereby each item determines one block of the decomposition.
The number of time periods is set to 3, 4, or 5.
For each of these 9 items-periods combinations 10 instances are generated resulting in a test set of 90 instances in total.
Moreover, the initial stock is fixed at zero to fulfill Assumption~\ref{assumption:DBproof}~\ref{assumption:DBproof:bounded}.
All remaining coefficients are randomly chosen positive integers,
whereby the coefficients $a^i$ are chosen from $\{2,3,4,5\}$, which gives a minimal $\Delta$-regularity in range $[2,60]$.

The instances of test set \testset{CFL} have 6, 12, or 18 clients and 2, 3, or 4 facilities.
For each of these 9 clients-facilities combinations 10 instances are generated resulting in a test set of 90 instances in total.
Each client specifies one block of the decomposition and one additional block is specified by all set-up variables $y_j$.
Moreover, we require single-sourcing for 50\% of the clients.
All coefficients are randomly chosen positive integers.
Thereby the coefficients $a^i$ are chosen from $\{2,3,4,5\}$, which gives a minimal $\Delta$-regularity in range $[2,60]$.

We preselected the instances such that for all 90 instances of both test sets holds:
a) all instances have an optimal solution,
b) at least one of the algorithm variants considered later finishes within the time limit.

These 180 instances inclusive the corresponding decomposition and the corresponding minimal $\Delta$-regularity are publicly available at \cite{githubrepo_decbranch}.

\subsection{Performance results}

We evaluate the behavior of both variants, $\Delta$DB and DB, on the two test sets.
For $\Delta$DB a minimal $\Delta$-regularity is determined according to Theorems~\ref{Theorem:MISL:Delta-regular} and~\ref{Theorem:CFL:Delta-regular}.
The initial DB is evaluated with four different tolerance parameters $\varepsilon$: $10^{-1}$, $10^{-2}$, $10^{-3}$, and $10^{-4}$.
So, in total we compare five variants of Decomposition Branching.
We always use a time limit of one hour to process an instance.


\begin{table}[h]
	\centering
	\setlength{\tabcolsep}{10pt}
	\begin{tabular}{lrrrrr}
		\toprule
		&            & \multicolumn{4}{c}{DB} \\
		\cmidrule(lr){3-6}
		state & $\Delta$DB & $10^{-1}$ & $10^{-2}$ & $10^{-3}$ & $10^{-4}$ \\
		\midrule
		timelimit &  38 &  13 &  52 &  60 &  71 \\
		finished\_opt &  52 &  6 &  2 &  2 &  2 \\
		finished\_nosol &  0 &  48 &  29 &  23 &  12 \\
		finished\_subopt &  0 &  23 &  7 &  5 &  5 \\
		\bottomrule
	\end{tabular}
	\caption{State of $\Delta$DB and DB with four different $\varepsilon$ for test set \testset{MISL}.}
	\label{tbl:run13:status:misl:exact}
\end{table}

\begin{table}[h]
	\centering
	\setlength{\tabcolsep}{10pt}
	\begin{tabular}{lrrrrr}
		\toprule
		&            & \multicolumn{4}{c}{DB} \\
		\cmidrule(lr){3-6}
		state & $\Delta$DB & $10^{-1}$ & $10^{-2}$ & $10^{-3}$ & $10^{-4}$ \\
		\midrule
		timelimit &  12 &  3 &  31 &  57 &  72 \\
		finished\_opt &  78 &  82 &  56 &  30 &  15 \\
		finished\_nosol &  0 &  2 &  2 &  2 &  2 \\
		finished\_subopt &  0 &  3 &  1 &  1 &  1 \\
		\bottomrule
	\end{tabular}
	\caption{State of $\Delta$DB and DB with four different $\varepsilon$ for test set \testset{CFL}.}
	\label{tbl:run13:status:cfl:exact}
\end{table}

First, we examine the state of all five variants for both test sets. 
These numbers are summarized in Table~\ref{tbl:run13:status:misl:exact} for test set \testset{MISL} and Table~\ref{tbl:run13:status:cfl:exact} for test set \testset{CFL}.
Every instance can terminate with one of four states.
If an algorithm variant reached the time limit of 1 hour, this instance is counted in row "timelimit".
Otherwise, the algorithm has finished, however, it can have three different states.
If the optimal solution was found, the instance is counted in row "finished\_opt".
If the algorithm terminated without finding any solution, it is counted in row "finished\_nosol" and is incorrectly marked as infeasible.
The last row, "finished\_subopt", counts all instances that terminated with a feasible but not optimal solution.

For \testset{MISL}, we observe that DB returns an incorrect state, either without solution or with suboptimal solution, for a considerable number of instances.
For example, with the largest $\varepsilon$ of $10^{-1}$, DB yields incorrect states for 71 of 90 instances.
The number of incorrect states decreases as $\varepsilon$ decreases.
DB finds the correct optimal solution for only 2 or 6 instances.
On the other hand, $\Delta$DB always returns a correct state.
It solves 52 instances optimally and reaches the time limit in 38 cases.

The algorithm variants applied to \testset{CFL} show a similar, albeit different, behavior.
DB finds an optimal solution for a large number of instances, however, this number decreases rapidly as $\varepsilon$ becomes smaller.
For example, with an $\varepsilon$ of $10^{-1}$, DB finds optimal solutions for 82 instances, but this number drops to 15 instances at an $\varepsilon$ of $10^{-4}$.
Some instances finished incorrectly without any solution or with suboptimal solution running DB, which also decreases with smaller $\varepsilon$.
In contrast, taking a look at $\Delta$DB, we observe that it does not yield any incorrect results.
At first glance, it seems strange that $\Delta$DB can solve fewer instances optimally than DB with $\varepsilon = 10^{-1}$.
However, this is not really surprising, as $\varepsilon = 10^{-1}$ is comparable to a $\Delta$ of $10$, and thus this variant can therefore make greater progress than $\Delta$DB with a $\Delta$ greater than $10$.
Nevertheless, $\Delta$DB does not have the risk of returning incorrect results.

\input{graphics/misl_plot}
\input{graphics/cfl_plot}

Figure~\ref{fig:nodes:misl:exact} and~\ref{fig:nodes:cfl:exact} visualize the correlation between the number of nodes used and the state.
On the horizontal axis the number of nodes used by $\Delta$DB is plotted and on the vertical axis the number of nodes used by one of the four variants of DB is plotted.
Thereby, the shapes specify the states of the algorithms.
For \testset{MISL} one can observe that $\Delta$DB needs considerably fewer nodes in most cases except when DB returns an incorrect result.
Moreover the number of nodes increases for DB when decreasing $\varepsilon$.
A similar effect can get observed for \testset{CFL}.
However, for \testset{CFL} the number of nodes used by $\Delta$DB is not necessarily smaller compared to the variant with $\varepsilon=10^{-1}$.
It is worth noting that the average cost for a single node is comparable across all variants, as the algorithms differ only in the way they reformulate strict inequalities.

In summary, the incorporation of $\Delta$-regularity in Decomposition Branching offers substantial improvements over the initial approach.
The theoretical drawbacks of the initial Decomposition Branching algorithm as stated in Remark~\ref{remark:DB:epsilon} were empirically validated with these experiments.
$\Delta$DB consistently returns correct results, whereas DB can return without any solution or with a suboptimal solution although the algorithm finished within the time limit.
Moreover, $\Delta$DB has usually smaller branching trees, that is, number of nodes, especially when comparing to variants with small $\varepsilon$.
Consequently, the utilization of $\Delta$-regularity in Decomposition Branching is highly recommended from both theoretical and practical points of view, and provides reliable results with usually less effort.

%% file: graphics/misl_plot.tex

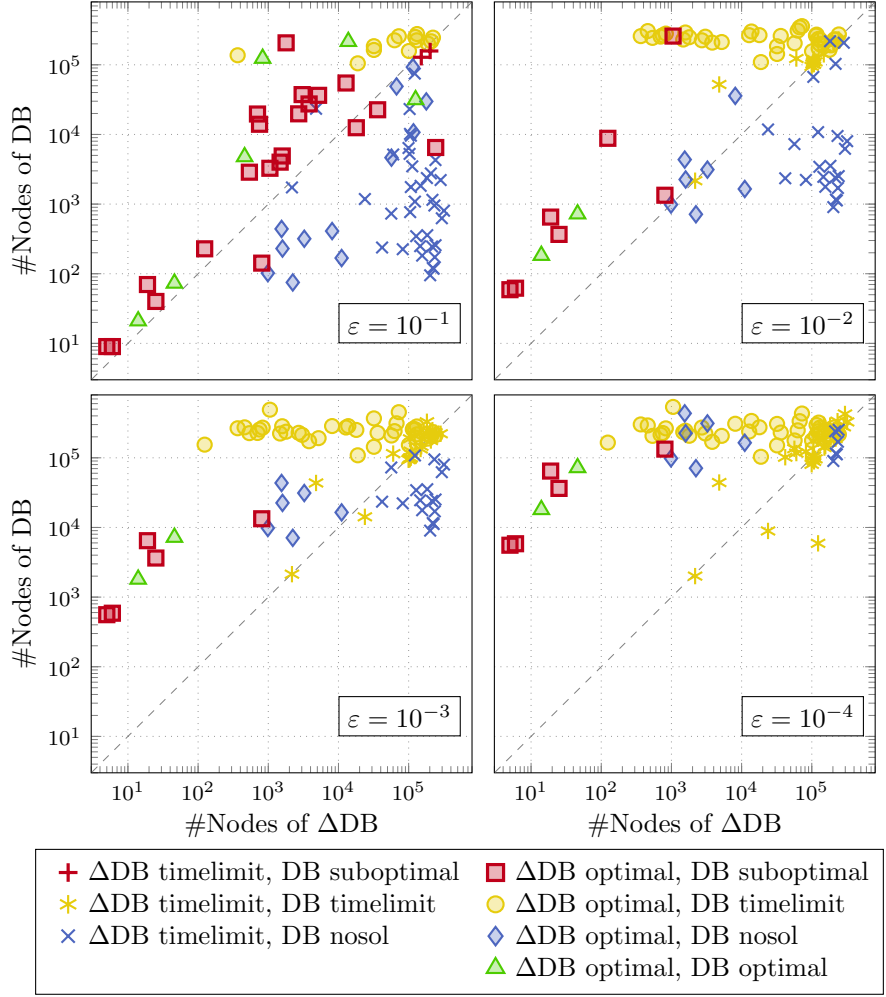
\begin{figure}[th]
	\centering
	\definecolor{colorinfeasplot}{rgb}{0.3,0.4,0.75} 
	\definecolor{colortimeplot}{rgb}{0.9,0.8,0.05} 
	\definecolor{coloroptplot}{rgb}{0.3,0.8,0.0} 
	\definecolor{colorsuboptplot}{rgb}{0.75,0.0,0.10} 
	\pgfplotsset{
		nodes plot/.style = {
			tick align=inside,
			ticklabel style = {font=\footnotesize},
			xmin=3,
			xmax=8e5,
			ymin=3,
			ymax=8e5,
			unit vector ratio={1 1},
			scale=0.88,
			grid=major,
			major grid style={line width=.3pt,draw=gray, dotted}
		},
		timetime addplot/.style = {
			draw=colortimeplot,
			fill=colortimeplot,
			mark=asterisk,
			only marks,
			mark size=3.5pt,
			thick
		},
		timeinfeas addplot/.style = {
			draw=colorinfeasplot,
			fill=colorinfeasplot,
			mark=x,
			only marks,
			mark size=3.5pt,
			thick
		},
		timesubopt addplot/.style = {
			draw=colorsuboptplot,
			fill=colorsuboptplot,
			mark=+,
			only marks,
			mark size=3.5pt,
			very thick
		},
		timeopt addplot/.style = {
			draw=coloroptplot,
			fill=coloroptplot,
			mark=Mercedes star flipped,
			only marks,
			mark size=3.5pt,
			very thick
		},
		optinfeas addplot/.style = {
			draw=colorinfeasplot,
			fill=colorinfeasplot,
			mark=diamond*,
			only marks,
			fill opacity=0.3,
			mark size=3.5pt,
			thick
		},
		optsubopt addplot/.style = {
			draw=colorsuboptplot,
			fill=colorsuboptplot,
			mark=square*,
			only marks,
			fill opacity=0.3,
			mark size=3pt,
			very thick
		},
		optopt addplot/.style = {
			draw=coloroptplot,
			fill=coloroptplot,
			mark=triangle*,
			only marks,
			fill opacity=0.3,
			mark size=3.8pt,
			thick
		},
		opttime addplot/.style = {
			draw=colortimeplot,
			fill=colortimeplot,
			mark=*,
			only marks,
			fill opacity=0.3,
			mark size=3pt,
			thick
		},
		discard if not/.style 2 args={
			filter discard warning=false,
			x filter/.append code={
				\edef\tempa{\thisrow{#1}}
				\edef\tempb{#2}
				\ifx\tempa\tempb
				\else
				\def\pgfmathresult{NaN}
				\fi
			},
		},
	}
	\begin{tikzpicture}
		\begin{loglogaxis}[nodes plot,
		ylabel={\#Nodes of DB},
		xticklabels=\empty,
		xtick pos=right,
		ytick pos= left]
		\draw[line width=.2pt,draw=gray, dashed] (1,1) -- (1e7,1e7);
		\addplot[timetime addplot]
		table[x=Nodes_delta,
		y=Nodes_eps110,
		col sep=comma,
		discard if not={Detailed_status_delta}{timelimit},
		discard if not={Detailed_status_eps110}{timelimit}]
		{graphics/df_misl_plot.csv};
		\addplot[opttime addplot]
		table[x=Nodes_delta,
		y=Nodes_eps110,
		col sep=comma,
		discard if not={Detailed_status_delta}{finished_opt},
		discard if not={Detailed_status_eps110}{timelimit}]
		{graphics/df_misl_plot.csv};
		\addplot[timeinfeas addplot]
		table[x=Nodes_delta,
		y=Nodes_eps110,
		col sep=comma,
		discard if not={Detailed_status_delta}{timelimit},
		discard if not={Detailed_status_eps110}{finished_infeas}]
		{graphics/df_misl_plot.csv};
		\addplot[optinfeas addplot]
		table[x=Nodes_delta,
		y=Nodes_eps110,
		col sep=comma,
		discard if not={Detailed_status_delta}{finished_opt},
		discard if not={Detailed_status_eps110}{finished_infeas}]
		{graphics/df_misl_plot.csv};
		\addplot[optsubopt addplot]
		table[x=Nodes_delta,
		y=Nodes_eps110,
		col sep=comma,
		discard if not={Detailed_status_delta}{finished_opt},
		discard if not={Detailed_status_eps110}{finished_subopt}]
		{graphics/df_misl_plot.csv};
		\addplot[optopt addplot]
		table[x=Nodes_delta,
		y=Nodes_eps110,
		col sep=comma,
		discard if not={Detailed_status_delta}{finished_opt},
		discard if not={Detailed_status_eps110}{finished_opt}]
		{graphics/df_misl_plot.csv};
		\addplot[timeopt addplot]
		table[x=Nodes_delta,
		y=Nodes_eps110,
		col sep=comma,
		discard if not={Detailed_status_delta}{timelimit},
		discard if not={Detailed_status_eps110}{finished_opt}]
		{graphics/df_misl_plot.csv};
		\addplot[timesubopt addplot]
		table[x=Nodes_delta,
		y=Nodes_eps110,
		col sep=comma,
		discard if not={Detailed_status_delta}{timelimit},
		discard if not={Detailed_status_eps110}{finished_subopt}]
		{graphics/df_misl_plot.csv};\label{ptimesubopt}
		
		\node[draw, fill=white, anchor=south west] at (axis cs:1e4,10) {$\varepsilon=10^{-1}$};
		\end{loglogaxis}
	\end{tikzpicture}
\hspace*{0pt}
	\begin{tikzpicture}
		\begin{loglogaxis}[nodes plot,
		xticklabels=\empty,
		yticklabels=\empty,
		xtick pos=right,
		ytick pos= right]
		\draw[line width=.2pt,draw=gray, dashed] (1,1) -- (1e7,1e7);
		\addplot[timetime addplot]
		table[x=Nodes_delta,
		y=Nodes_eps1100,
		col sep=comma,
		discard if not={Detailed_status_delta}{timelimit},
		discard if not={Detailed_status_eps1100}{timelimit}]
		{graphics/df_misl_plot.csv};
		\addplot[opttime addplot]
		table[x=Nodes_delta,
		y=Nodes_eps1100,
		col sep=comma,
		discard if not={Detailed_status_delta}{finished_opt},
		discard if not={Detailed_status_eps1100}{timelimit}]
		{graphics/df_misl_plot.csv};
		\addplot[timeinfeas addplot]
		table[x=Nodes_delta,
		y=Nodes_eps1100,
		col sep=comma,
		discard if not={Detailed_status_delta}{timelimit},
		discard if not={Detailed_status_eps1100}{finished_infeas}]
		{graphics/df_misl_plot.csv};
		\addplot[optinfeas addplot]
		table[x=Nodes_delta,
		y=Nodes_eps1100,
		col sep=comma,
		discard if not={Detailed_status_delta}{finished_opt},
		discard if not={Detailed_status_eps1100}{finished_infeas}]
		{graphics/df_misl_plot.csv};
		\addplot[optsubopt addplot]
		table[x=Nodes_delta,
		y=Nodes_eps1100,
		col sep=comma,
		discard if not={Detailed_status_delta}{finished_opt},
		discard if not={Detailed_status_eps1100}{finished_subopt}]
		{graphics/df_misl_plot.csv};
		\addplot[optopt addplot]
		table[x=Nodes_delta,
		y=Nodes_eps1100,
		col sep=comma,
		discard if not={Detailed_status_delta}{finished_opt},
		discard if not={Detailed_status_eps1100}{finished_opt}]
		{graphics/df_misl_plot.csv};
		\addplot[timeopt addplot]
		table[x=Nodes_delta,
		y=Nodes_eps1100,
		col sep=comma,
		discard if not={Detailed_status_delta}{timelimit},
		discard if not={Detailed_status_eps1100}{finished_opt}]
		{graphics/df_misl_plot.csv};
		\addplot[timesubopt addplot]
		table[x=Nodes_delta,
		y=Nodes_eps1100,
		col sep=comma,
		discard if not={Detailed_status_delta}{timelimit},
		discard if not={Detailed_status_eps1100}{finished_subopt}]
		{graphics/df_misl_plot.csv};
		
		\node[draw, fill=white, anchor=south west] at (axis cs:1e4,10) {$\varepsilon=10^{-2}$};
		\end{loglogaxis}
	\end{tikzpicture}

\vspace*{3pt}

	\begin{tikzpicture}
		\begin{loglogaxis}[nodes plot,
		ylabel={\#Nodes of DB},
		xlabel={\#Nodes of $\Delta$DB},
		tick pos=left]
		\draw[line width=.2pt,draw=gray, dashed] (1,1) -- (1e7,1e7);
		\addplot[timetime addplot]
		table[x=Nodes_delta,
		y=Nodes_eps11000,
		col sep=comma,
		discard if not={Detailed_status_delta}{timelimit},
		discard if not={Detailed_status_eps11000}{timelimit}]
		{graphics/df_misl_plot.csv};
		\addplot[opttime addplot]
		table[x=Nodes_delta,
		y=Nodes_eps11000,
		col sep=comma,
		discard if not={Detailed_status_delta}{finished_opt},
		discard if not={Detailed_status_eps11000}{timelimit}]
		{graphics/df_misl_plot.csv};
		\addplot[timeinfeas addplot]
		table[x=Nodes_delta,
		y=Nodes_eps11000,
		col sep=comma,
		discard if not={Detailed_status_delta}{timelimit},
		discard if not={Detailed_status_eps11000}{finished_infeas}]
		{graphics/df_misl_plot.csv};
		\addplot[optinfeas addplot]
		table[x=Nodes_delta,
		y=Nodes_eps11000,
		col sep=comma,
		discard if not={Detailed_status_delta}{finished_opt},
		discard if not={Detailed_status_eps11000}{finished_infeas}]
		{graphics/df_misl_plot.csv};
		\addplot[optsubopt addplot]
		table[x=Nodes_delta,
		y=Nodes_eps11000,
		col sep=comma,
		discard if not={Detailed_status_delta}{finished_opt},
		discard if not={Detailed_status_eps11000}{finished_subopt}]
		{graphics/df_misl_plot.csv};
		\addplot[optopt addplot]
		table[x=Nodes_delta,
		y=Nodes_eps11000,
		col sep=comma,
		discard if not={Detailed_status_delta}{finished_opt},
		discard if not={Detailed_status_eps11000}{finished_opt}]
		{graphics/df_misl_plot.csv};
		\addplot[timeopt addplot]
		table[x=Nodes_delta,
		y=Nodes_eps11000,
		col sep=comma,
		discard if not={Detailed_status_delta}{timelimit},
		discard if not={Detailed_status_eps11000}{finished_opt}]
		{graphics/df_misl_plot.csv};
		\addplot[timesubopt addplot]
		table[x=Nodes_delta,
		y=Nodes_eps11000,
		col sep=comma,
		discard if not={Detailed_status_delta}{timelimit},
		discard if not={Detailed_status_eps11000}{finished_subopt}]
		{graphics/df_misl_plot.csv};
		
		\node[draw, fill=white, anchor=south west] at (axis cs:1e4,10) {$\varepsilon=10^{-3}$};
		\end{loglogaxis}
	\end{tikzpicture}
\hspace*{0pt}
	\begin{tikzpicture}
		\begin{loglogaxis}[nodes plot,
		yticklabels=\empty,
		xlabel={\#Nodes of $\Delta$DB},
		xtick pos=left,
		ytick pos=right]
		\draw[line width=.2pt,draw=gray, dashed] (1,1) -- (1e7,1e7);
		\addplot[timetime addplot]
		table[x=Nodes_delta,
		y=Nodes_eps110000,
		col sep=comma,
		discard if not={Detailed_status_delta}{timelimit},
		discard if not={Detailed_status_eps110000}{timelimit}]
		{graphics/df_misl_plot.csv};\label{ptimetime}
		\addplot[opttime addplot]
		table[x=Nodes_delta,
		y=Nodes_eps110000,
		col sep=comma,
		discard if not={Detailed_status_delta}{finished_opt},
		discard if not={Detailed_status_eps110000}{timelimit}]
		{graphics/df_misl_plot.csv};\label{popttime}
		\addplot[timeinfeas addplot]
		table[x=Nodes_delta,
		y=Nodes_eps110000,
		col sep=comma,
		discard if not={Detailed_status_delta}{timelimit},
		discard if not={Detailed_status_eps110000}{finished_infeas}]
		{graphics/df_misl_plot.csv};\label{ptimeinfeas}
		\addplot[optinfeas addplot]
		table[x=Nodes_delta,
		y=Nodes_eps110000,
		col sep=comma,
		discard if not={Detailed_status_delta}{finished_opt},
		discard if not={Detailed_status_eps110000}{finished_infeas}]
		{graphics/df_misl_plot.csv};\label{poptinfeas}
		\addplot[optsubopt addplot]
		table[x=Nodes_delta,
		y=Nodes_eps110000,
		col sep=comma,
		discard if not={Detailed_status_delta}{finished_opt},
		discard if not={Detailed_status_eps110000}{finished_subopt}]
		{graphics/df_misl_plot.csv};\label{poptsubopt}
		\addplot[optopt addplot]
		table[x=Nodes_delta,
		y=Nodes_eps110000,
		col sep=comma,
		discard if not={Detailed_status_delta}{finished_opt},
		discard if not={Detailed_status_eps110000}{finished_opt}]
		{graphics/df_misl_plot.csv};\label{poptopt}
		\addplot[timeopt addplot]
		table[x=Nodes_delta,
		y=Nodes_eps110000,
		col sep=comma,
		discard if not={Detailed_status_delta}{timelimit},
		discard if not={Detailed_status_eps110000}{finished_opt}]
		{graphics/df_misl_plot.csv};
		\addplot[timesubopt addplot]
		table[x=Nodes_delta,
		y=Nodes_eps110000,
		col sep=comma,
		discard if not={Detailed_status_delta}{timelimit},
		discard if not={Detailed_status_eps110000}{finished_subopt}]
		{graphics/df_misl_plot.csv};
		
		\node[draw, fill=white, anchor=south west] at (axis cs:1e4,10) {$\varepsilon=10^{-4}$};
		\end{loglogaxis}
	\end{tikzpicture}

	\hspace*{5pt}
	\begin{tikzpicture}
	\node [draw=black,fill=none, anchor=east] at (0,0)
	{\setlength{\tabcolsep}{2pt}
	\begin{tabular}{cl@{\hskip 10pt}cl}
	\ref{ptimesubopt} & $\Delta$DB timelimit, DB suboptimal & \ref{poptsubopt} & $\Delta$DB optimal, DB suboptimal \\
	\ref{ptimetime} & $\Delta$DB timelimit, DB timelimit & \ref{popttime} & $\Delta$DB optimal, DB timelimit \\
	\ref{ptimeinfeas} & $\Delta$DB timelimit, DB nosol & \ref{poptinfeas} & $\Delta$DB optimal, DB nosol \\
	& & \ref{poptopt} & $\Delta$DB optimal, DB optimal \\
	\end{tabular}};
	\end{tikzpicture}

	\caption{Number of nodes used by $\Delta$DB compared to DB with four different $\varepsilon$ for test set \testset{MISL}. Shapes define the state of the two compared variants.}
	\label{fig:nodes:misl:exact}
\end{figure}

%% file: graphics/cfl_plot.tex

\begin{figure}[th]
	\centering
	\definecolor{colorinfeasplot}{rgb}{0.3,0.4,0.75} 
	\definecolor{colortimeplot}{rgb}{0.9,0.8,0.05} 
	\definecolor{coloroptplot}{rgb}{0.3,0.8,0.0} 
	\definecolor{colorsuboptplot}{rgb}{0.75,0.0,0.10} 
	\pgfplotsset{
		nodes plot/.style = {
			tick align=inside,
			ticklabel style = {font=\footnotesize},
			xmin=0.4,
			xmax=9e5,
			ymin=0.4,
			ymax=9e5,
			unit vector ratio={1 1},
			scale=0.88,
			grid=major,
			major grid style={line width=.3pt,draw=gray, dotted}
		},
		timetime addplot/.style = {
			draw=colortimeplot,
			fill=colortimeplot,
			mark=asterisk,
			only marks,
			mark size=3.5pt,
			thick
		},
		timeinfeas addplot/.style = {
			draw=colorinfeasplot,
			fill=colorinfeasplot,
			mark=x,
			only marks,
			mark size=3.5pt,
			thick
		},
		timesubopt addplot/.style = {
			draw=colorsuboptplot,
			fill=colorsuboptplot,
			mark=+,
			only marks,
			mark size=3.5pt,
			very thick
		},
		timeopt addplot/.style = {
			draw=coloroptplot,
			fill=coloroptplot,
			mark=Mercedes star flipped,
			only marks,
			mark size=3.5pt,
			very thick
		},
		optinfeas addplot/.style = {
			draw=colorinfeasplot,
			fill=colorinfeasplot,
			mark=diamond*,
			only marks,
			fill opacity=0.3,
			mark size=3.5pt,
			thick
		},
		optsubopt addplot/.style = {
			draw=colorsuboptplot,
			fill=colorsuboptplot,
			mark=square*,
			only marks,
			fill opacity=0.3,
			mark size=3pt,
			very thick
		},
		optopt addplot/.style = {
			draw=coloroptplot,
			fill=coloroptplot,
			mark=triangle*,
			only marks,
			fill opacity=0.3,
			mark size=3.8pt,
			thick
		},
		opttime addplot/.style = {
			draw=colortimeplot,
			fill=colortimeplot,
			mark=*,
			only marks,
			fill opacity=0.3,
			mark size=3pt,
			thick
		},
		discard if not/.style 2 args={
			filter discard warning=false,
			x filter/.append code={
				\edef\tempa{\thisrow{#1}}
				\edef\tempb{#2}
				\ifx\tempa\tempb
				\else
				\def\pgfmathresult{NaN}
				\fi
			},
		},
	}
	\begin{tikzpicture}
		\begin{loglogaxis}[nodes plot,
		ylabel={\#Nodes of DB},
		xticklabels=\empty,
		xtick pos=right,
		ytick pos= left]
		\draw[line width=.2pt,draw=gray, dashed] (0.4,00.4) -- (1e7,1e7);
		\addplot[timetime addplot]
		table[x=Nodes_delta,
		y=Nodes_eps110,
		col sep=comma,
		discard if not={Detailed_status_delta}{timelimit},
		discard if not={Detailed_status_eps110}{timelimit}]
		{graphics/df_cfl_plot.csv};
		\addplot[opttime addplot]
		table[x=Nodes_delta,
		y=Nodes_eps110,
		col sep=comma,
		discard if not={Detailed_status_delta}{finished_opt},
		discard if not={Detailed_status_eps110}{timelimit}]
		{graphics/df_cfl_plot.csv};
		\addplot[timeinfeas addplot]
		table[x=Nodes_delta,
		y=Nodes_eps110,
		col sep=comma,
		discard if not={Detailed_status_delta}{timelimit},
		discard if not={Detailed_status_eps110}{finished_infeas}]
		{graphics/df_cfl_plot.csv};
		\addplot[optinfeas addplot]
		table[x=Nodes_delta,
		y=Nodes_eps110,
		col sep=comma,
		discard if not={Detailed_status_delta}{finished_opt},
		discard if not={Detailed_status_eps110}{finished_infeas}]
		{graphics/df_cfl_plot.csv};
		\addplot[optopt addplot]
		table[x=Nodes_delta,
		y=Nodes_eps110,
		col sep=comma,
		discard if not={Detailed_status_delta}{finished_opt},
		discard if not={Detailed_status_eps110}{finished_opt}]
		{graphics/df_cfl_plot.csv};
		\addplot[optsubopt addplot]
		table[x=Nodes_delta,
		y=Nodes_eps110,
		col sep=comma,
		discard if not={Detailed_status_delta}{finished_opt},
		discard if not={Detailed_status_eps110}{finished_subopt}]
		{graphics/df_cfl_plot.csv};
		\addplot[timeopt addplot]
		table[x=Nodes_delta,
		y=Nodes_eps110,
		col sep=comma,
		discard if not={Detailed_status_delta}{timelimit},
		discard if not={Detailed_status_eps110}{finished_opt}]
		{graphics/df_cfl_plot.csv};\label{ptimeopt}
		\addplot[timesubopt addplot]
		table[x=Nodes_delta,
		y=Nodes_eps110,
		col sep=comma,
		discard if not={Detailed_status_delta}{timelimit},
		discard if not={Detailed_status_eps110}{finished_subopt}]
		{graphics/df_cfl_plot.csv};
		
		\node[draw, fill=white, anchor=south west] at (axis cs:5e3,1) {$\varepsilon=10^{-1}$};
		\end{loglogaxis}
	\end{tikzpicture}
\hspace*{0pt}
	\begin{tikzpicture}
		\begin{loglogaxis}[nodes plot,
		xticklabels=\empty,
		yticklabels=\empty,
		xtick pos=right,
		ytick pos= right]
		\draw[line width=.2pt,draw=gray, dashed] (0.4,00.4) -- (1e7,1e7);
		\addplot[timetime addplot]
		table[x=Nodes_delta,
		y=Nodes_eps1100,
		col sep=comma,
		discard if not={Detailed_status_delta}{timelimit},
		discard if not={Detailed_status_eps1100}{timelimit}]
		{graphics/df_cfl_plot.csv};
		\addplot[opttime addplot]
		table[x=Nodes_delta,
		y=Nodes_eps1100,
		col sep=comma,
		discard if not={Detailed_status_delta}{finished_opt},
		discard if not={Detailed_status_eps1100}{timelimit}]
		{graphics/df_cfl_plot.csv};
		\addplot[timeinfeas addplot]
		table[x=Nodes_delta,
		y=Nodes_eps1100,
		col sep=comma,
		discard if not={Detailed_status_delta}{timelimit},
		discard if not={Detailed_status_eps1100}{finished_infeas}]
		{graphics/df_cfl_plot.csv};
		\addplot[optinfeas addplot]
		table[x=Nodes_delta,
		y=Nodes_eps1100,
		col sep=comma,
		discard if not={Detailed_status_delta}{finished_opt},
		discard if not={Detailed_status_eps1100}{finished_infeas}]
		{graphics/df_cfl_plot.csv};
		\addplot[optopt addplot]
		table[x=Nodes_delta,
		y=Nodes_eps1100,
		col sep=comma,
		discard if not={Detailed_status_delta}{finished_opt},
		discard if not={Detailed_status_eps1100}{finished_opt}]
		{graphics/df_cfl_plot.csv};
		\addplot[optsubopt addplot]
		table[x=Nodes_delta,
		y=Nodes_eps1100,
		col sep=comma,
		discard if not={Detailed_status_delta}{finished_opt},
		discard if not={Detailed_status_eps1100}{finished_subopt}]
		{graphics/df_cfl_plot.csv};
		\addplot[timeopt addplot]
		table[x=Nodes_delta,
		y=Nodes_eps1100,
		col sep=comma,
		discard if not={Detailed_status_delta}{timelimit},
		discard if not={Detailed_status_eps1100}{finished_opt}]
		{graphics/df_cfl_plot.csv};
		\addplot[timesubopt addplot]
		table[x=Nodes_delta,
		y=Nodes_eps1100,
		col sep=comma,
		discard if not={Detailed_status_delta}{timelimit},
		discard if not={Detailed_status_eps1100}{finished_subopt}]
		{graphics/df_cfl_plot.csv};
		
		\node[draw, fill=white, anchor=south west] at (axis cs:5e3,1) {$\varepsilon=10^{-2}$};
		\end{loglogaxis}
	\end{tikzpicture}

\vspace*{3pt}

	\begin{tikzpicture}
		\begin{loglogaxis}[nodes plot,
		ylabel={\#Nodes of DB},
		xlabel={\#Nodes of $\Delta$DB},
		tick pos=left]
		\draw[line width=.2pt,draw=gray, dashed] (0.4,00.4) -- (1e7,1e7);
		\addplot[timetime addplot]
		table[x=Nodes_delta,
		y=Nodes_eps11000,
		col sep=comma,
		discard if not={Detailed_status_delta}{timelimit},
		discard if not={Detailed_status_eps11000}{timelimit}]
		{graphics/df_cfl_plot.csv};
		\addplot[opttime addplot]
		table[x=Nodes_delta,
		y=Nodes_eps11000,
		col sep=comma,
		discard if not={Detailed_status_delta}{finished_opt},
		discard if not={Detailed_status_eps11000}{timelimit}]
		{graphics/df_cfl_plot.csv};
		\addplot[timeinfeas addplot]
		table[x=Nodes_delta,
		y=Nodes_eps11000,
		col sep=comma,
		discard if not={Detailed_status_delta}{timelimit},
		discard if not={Detailed_status_eps11000}{finished_infeas}]
		{graphics/df_cfl_plot.csv};
		\addplot[optinfeas addplot]
		table[x=Nodes_delta,
		y=Nodes_eps11000,
		col sep=comma,
		discard if not={Detailed_status_delta}{finished_opt},
		discard if not={Detailed_status_eps11000}{finished_infeas}]
		{graphics/df_cfl_plot.csv};
		\addplot[optopt addplot]
		table[x=Nodes_delta,
		y=Nodes_eps11000,
		col sep=comma,
		discard if not={Detailed_status_delta}{finished_opt},
		discard if not={Detailed_status_eps11000}{finished_opt}]
		{graphics/df_cfl_plot.csv};
		\addplot[optsubopt addplot]
		table[x=Nodes_delta,
		y=Nodes_eps11000,
		col sep=comma,
		discard if not={Detailed_status_delta}{finished_opt},
		discard if not={Detailed_status_eps11000}{finished_subopt}]
		{graphics/df_cfl_plot.csv};
		\addplot[timeopt addplot]
		table[x=Nodes_delta,
		y=Nodes_eps11000,
		col sep=comma,
		discard if not={Detailed_status_delta}{timelimit},
		discard if not={Detailed_status_eps11000}{finished_opt}]
		{graphics/df_cfl_plot.csv};
		\addplot[timesubopt addplot]
		table[x=Nodes_delta,
		y=Nodes_eps11000,
		col sep=comma,
		discard if not={Detailed_status_delta}{timelimit},
		discard if not={Detailed_status_eps11000}{finished_subopt}]
		{graphics/df_cfl_plot.csv};
		
		\node[draw, fill=white, anchor=south west] at (axis cs:5e3,1) {$\varepsilon=10^{-3}$};
		\end{loglogaxis}
	\end{tikzpicture}
\hspace*{0pt}
	\begin{tikzpicture}
		\begin{loglogaxis}[nodes plot,
		yticklabels=\empty,
		xlabel={\#Nodes of $\Delta$DB},
		xtick pos=left,
		ytick pos=right]
		\draw[line width=.2pt,draw=gray, dashed] (0.4,00.4) -- (1e7,1e7);
		\addplot[timetime addplot]
		table[x=Nodes_delta,
		y=Nodes_eps110000,
		col sep=comma,
		discard if not={Detailed_status_delta}{timelimit},
		discard if not={Detailed_status_eps110000}{timelimit}]
		{graphics/df_cfl_plot.csv};
		\addplot[opttime addplot]
		table[x=Nodes_delta,
		y=Nodes_eps110000,
		col sep=comma,
		discard if not={Detailed_status_delta}{finished_opt},
		discard if not={Detailed_status_eps110000}{timelimit}]
		{graphics/df_cfl_plot.csv};
		\addplot[timeinfeas addplot]
		table[x=Nodes_delta,
		y=Nodes_eps110000,
		col sep=comma,
		discard if not={Detailed_status_delta}{timelimit},
		discard if not={Detailed_status_eps110000}{finished_infeas}]
		{graphics/df_cfl_plot.csv};
		\addplot[optinfeas addplot]
		table[x=Nodes_delta,
		y=Nodes_eps110000,
		col sep=comma,
		discard if not={Detailed_status_delta}{finished_opt},
		discard if not={Detailed_status_eps110000}{finished_infeas}]
		{graphics/df_cfl_plot.csv};
		\addplot[optopt addplot]
		table[x=Nodes_delta,
		y=Nodes_eps110000,
		col sep=comma,
		discard if not={Detailed_status_delta}{finished_opt},
		discard if not={Detailed_status_eps110000}{finished_opt}]
		{graphics/df_cfl_plot.csv};
		\addplot[optsubopt addplot]
		table[x=Nodes_delta,
		y=Nodes_eps110000,
		col sep=comma,
		discard if not={Detailed_status_delta}{finished_opt},
		discard if not={Detailed_status_eps110000}{finished_subopt}]
		{graphics/df_cfl_plot.csv};
		\addplot[timeopt addplot]
		table[x=Nodes_delta,
		y=Nodes_eps110000,
		col sep=comma,
		discard if not={Detailed_status_delta}{timelimit},
		discard if not={Detailed_status_eps110000}{finished_opt}]
		{graphics/df_cfl_plot.csv};
		\addplot[timesubopt addplot]
		table[x=Nodes_delta,
		y=Nodes_eps110000,
		col sep=comma,
		discard if not={Detailed_status_delta}{timelimit},
		discard if not={Detailed_status_eps110000}{finished_subopt}]
		{graphics/df_cfl_plot.csv};

		\node[draw, fill=white, anchor=south west] at (axis cs:5e3,1) {$\varepsilon=10^{-4}$};
		\end{loglogaxis}
	\end{tikzpicture}

	\hspace*{5pt}
	\begin{tikzpicture}
	\node [draw=black,fill=none, anchor=east] at (0,0)
	{\setlength{\tabcolsep}{2pt}
	\begin{tabular}{cl@{\hskip 10pt}cl}
	\ref{ptimesubopt} & $\Delta$DB timelimit, DB suboptimal & \ref{poptsubopt} & $\Delta$DB optimal, DB suboptimal \\
	\ref{ptimetime} & $\Delta$DB timelimit, DB timelimit & \ref{popttime} & $\Delta$DB optimal, DB timelimit \\
	\ref{ptimeopt} & $\Delta$DB timelimit, DB optimal & \ref{poptopt} & $\Delta$DB optimal, DB optimal \\
	& & \ref{poptinfeas} & $\Delta$DB optimal, DB nosol  \\
	\end{tabular}};
\end{tikzpicture}

\caption{Number of nodes used by $\Delta$DB compared to DB with four different $\varepsilon$ for test set \testset{CFL}. Shapes define the state of the two compared variants.}
\label{fig:nodes:cfl:exact}
\end{figure}
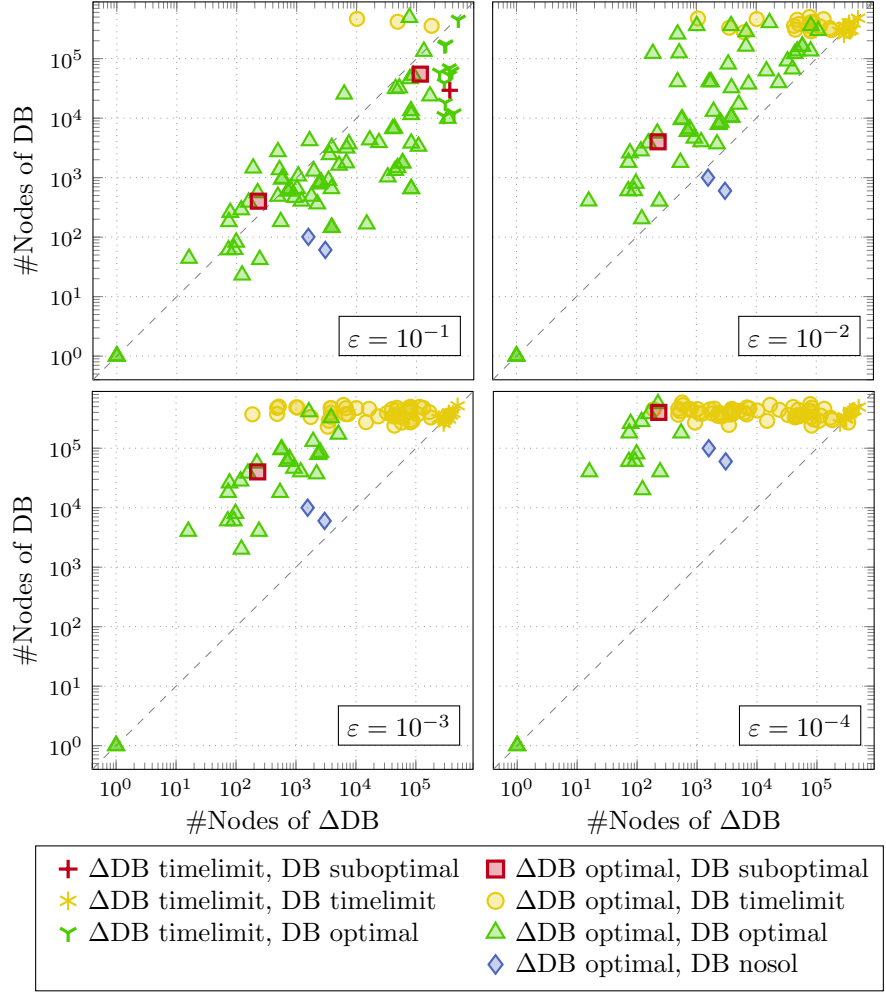

%% file: section_conclusion.tex
\section{Conclusion}
\label{sect:conclusion}

In this paper we present a rounding procedure for strict inequalities and apply it to the Decomposition Branching approach.
The procedure is based on lattices containing at least one optimal solution,
whereby we restrict ourselves to lattices determined by the $\Delta$-regularity of the constraint matrix $A$ corresponding to the continuous variables.
Our enhanced version of Decomposition Branching guaranties to return an exact optimal solution, whereas the initial variant risks to cut off all optimal solutions.
These theoretical findings are underpinned by a computational study.

The performance of the enhanced version of decomposition branching is considerably better than the initial Decomposition Branching,
but unfortunately neither the initial nor the extended version can keep up with state-of-the-art solvers for the MISL and CFL instances at hand.
Nevertheless, we believe that the enhanced version of Decomposition Branching can be a powerful algorithm if the model is chosen appropriately and if
for the initial Decomposition Branching an appropriate and efficient implementation is available.

Our rounding procedure can be applied to all approaches that use strict inequalities for mixed-integer linear problems with known (minimal) $\Delta$-regularity.
For example, it would be particularly interesting to investigate the extent to which $\Delta$-regularity can be beneficial inside variants
of Benders decomposition as proposed in \cite[Section~5.4]{Weninger2016} and cutting plane methods as projected Chv\'{a}tal-Gomory cuts~\cite{Bonami2008}.